\documentclass[12pt]{amsart}
\usepackage{amsmath,amssymb,amscd,latexsym,amsthm,mathrsfs,amsfonts}
\usepackage[unicode]{hyperref}
\usepackage{fancyhdr}
\ifx\pdfoutput\undefined\else
\hypersetup{
colorlinks=true,
linkcolor=blue,
citecolor=blue,
urlcolor=blue,
filecolor=blue,
bookmarksnumbered=true,
pdfstartview=FitH,
pdfhighlight=/N
}
\fi
\usepackage{pgf,tikz}

\usepackage{array}
\allowdisplaybreaks[1]

\newtheorem{thm}{Theorem}[section] 

\newtheorem{prop}[thm]{Proposition} 

\theoremstyle{definition}

\usepackage{cite}

\renewcommand{\author}[1]{\large\rm #1\\ \bigskip}
\renewcommand{\title}[1]{\bigskip\bigskip\Large\bf #1\bigskip\bigskip\\}

\begin{document}


\vglue .3 cm

\begin{center}

\title{Counting Planar Eulerian Orientations.}
\author{Andrew Elvey-Price\footnote[1]{email: andrewelveyprice@gmail.com} and Anthony J Guttmann\footnote[2]{email: guttmann@unimelb.edu.au} }
\address{School of Mathematics and Statistics\\
The University of Melbourne, Victoria 3010, Australia}
\end{center}
\setcounter{footnote}{0}

\begin{abstract}
Inspired by the paper of Bonichon, Bousquet-M\'elou, Dorbec and Pennarun \cite{BBDP16}, we give a system of functional equations which characterise the ordinary generating function, $U(x),$ for the number of planar Eulerian orientations counted by edges. We also characterise the ogf $A(x)$, for 4-valent planar Eulerian orientations counted by vertices in a similar way. The latter problem is equivalent to the 6-vertex problem on a random lattice, widely studied in mathematical physics. While unable to solve these functional equations, they immediately provide polynomial-time algorithms for computing the coefficients of the generating function. From these algorithms we have obtained 100 terms for $U(x)$ and 90 terms for $A(x).$

Analysis of these series suggests that they both behave as $const\cdot (1 - \mu x)/\log(1 - \mu x),$ where we conjecture that $\mu = 4\pi$ for Eulerian orientations counted by edges and $\mu=4\sqrt{3}\pi$ for 4-valent Eulerian orientations counted by vertices.

\end{abstract}
\makeatletter
\@setabstract
\makeatother


\section{Introduction}

Recently the problem of enumerating planar Eulerian orientations with $n$ edges was considered by Bonichon, Bousquet-M\'elou, Dorbec and Pennarun \cite{BBDP16}. Enumeration of Eulerian orientations on a given graph has been previously considered, for example by Felsner and Zickfeld \cite {FZ08} who established rigorous bounds on the growth constant for these and other combinatorial structures. The generating function for the number of rooted planar Eulerian maps\footnote {These are planar maps in which the degree of every vertex is even.} has been known since 1963 \cite{T63}. Indeed it is algebraic, and is just
$$M(t)= \frac{8t^2+12t-1+(1-8t)^{3/2}}{32t^2}.$$ As pointed out by Bonichon et al., planar maps with additional structure are much studied in both enumerative combinatorics and mathematical physics, and they give several examples. They then focus on Eulerian orientations, which restrict vertices to have equal in-degree and out-degree, and consider two classes of Eulerian orientations; the general class, counted by edges, and 4-valent Eulerian orientations counted by vertices. The latter is in the universality class of the celebrated six-vertex model on a random lattice, a problem that has been studied by Kostov \cite{K00} and Zinn-Justin \cite{ZJ00}. Unfortunately, the nature of their solutions are not in the form of a generating function that can be compared to the enumerative results of Bonichon {\em et al}, or our own more extensive enumerations. However we would expect the structure of the generating functions to be the same. Kostov gives the logarithm of the partition function as $$\log {Z} \sim \frac{c (T - T_c)^2}{\log(T-T_c)}.$$ The generating function  for the corresponding Eulerian maps should correspond to the derivative of $\log{Z},$ so the dominant term should behave as ${(T - T_c)}/{\log(T-T_c)}.$

  Let $U(x)$ be the generating function for planar Eulerian orientations, counted by edges. In this paper we find a system of functional equations which characterises the generating function $U(x)$. Similarly, we find a system of functional equations characterising the generating function $A(x)$ for 4-valent planar Eulerian orientations, counted by vertices. For each problem, these functional equations give rise to a polynomial time algorithm for computing the coefficients. 
  
Using these algorithms we have computed the first 90 coefficients of the generating function $A(x)$ and the first 100 coefficients of the generating function $U(x)$. In the final section we study these series and find that they behave as $const.  (1 - \mu x)/\log(1 - \mu x),$ where we conjecture that $\mu = 4\pi = 12.5663\ldots$ for Eulerian orientations counted by edges and $\mu=4\sqrt{3}\pi$ for 4-valent Eulerian orientations counted by vertices.

  In \cite{BBDP16}, a different approach was taken. Families of subsets and supersets were counted. These sets were indexed by a parameter $k,$ and the subsets and supersets were found to have algebraic generating functions. However the calculational difficulty increased with $k,$ so that $k=5$ was as far as they could go. With this data they calculated the generating function to 15 terms, and obtained  bounds for the growth rate $\mu$ for Eulerian orientations counted by edges, $11.22 < \mu <13.047,$ and gave the estimate $\mu \approx 12.5.$

\section{Functional equations}

In this section we derive a system of functional equations which characterise the ordinary generating functions $A(x)$ and $U(x)$ for 4-valent rooted planar Eulerian orientations counted by vertices and for rooted planar Eulerian orientations counted by edges respectively.

Recall that a {\em planar map} is a connected graph embedded on a sphere (multiple edges and loops are permitted, but edge crossings are not). A map is {\em rooted} if one of its edges is both oriented and distinguished. We call this edge the {\em root edge} and we call its source vertex the {\em root vertex}. In the following we will consider planar maps to be embedded in the plane, rather than the sphere, using the convention that the face to the left of the root edge is the outer face. Using this convention allows us to distinguish between the inner faces and the outer face of any connected subgraph of a rooted planar map.

  A {\em rooted planar Eulerian orientation} is a rooted planar map in which each edge is directed and each vertex has equal in-degree and out-degree. The direction of the root edge is not required to match the direction assigned when rooting the map.



\begin{figure}[ht]
\setlength{\captionindent}{0pt}
   \includegraphics[width=0.97\linewidth]{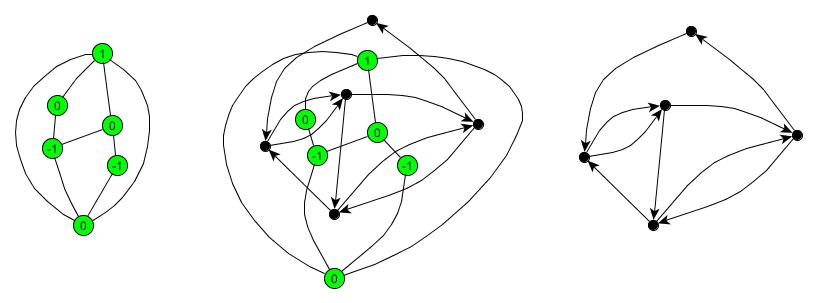} 
   \caption{An example of the transformation between an $N$-map (left of diagram) and the corresponding Eulerian orientation (right of diagram).}
   \label{fig:Nmap_to_orientation}
\end{figure}

\begin{prop} For any positive integer $n$, the number of $N$-maps with $n$ edges is equal to the number of rooted planar Eulerian orientations with $n$ edges. Also, the number of $N$-maps with $n$ edges, where each face has degree 4 is equal to the number of 4-valent rooted planar Eulerian orientations with $n$ edges.\end{prop}
\begin{proof}
Given an $N$-map, we can construct a directed map by orienting each edge from the lower number to the higher number. Then around each face, the number of clockwise edges is equal to the number of anticlockwise edges. Hence the dual of this map (where the orientations of the edges are defined by rotating the original edges $90^{\circ}$ clockwise) is an Eulerian orientation. By reversing each of these steps, we see that this transformation is a bijection. Hence, the number of $N$-maps with $n$ edges is equal to the number of rooted planar Eulerian orientations with $n$ edges. Using the same bijection, we see that the number of 4-valent rooted planar Eulerian orientations with $n$ edges is equal to the number of $N$-maps with $n$ edges, where each face has degree 4.
\end{proof}

We will occasionally refer to the height of an edge or a corner of an $N$-map. An edge is said to be at height $m+1/2$ if it joins a vertex at height $m$ to a vertex at height $m+1$. The height of a corner is simply equal to the height of the vertex which it contains. For any integer $k\geq0$, and any $N$-map $\Gamma$, let $\Sigma_{k}(\Gamma)$ be the subgraph of $\Gamma$ defined by taking only the vertices and edges of $\Gamma$ at height at least $k$. A $(\geq k)$-component of $\Gamma$ is a connected component of $\Sigma_{k}(\Gamma)$. For any $(\geq k)$-component $\tau$ of $\Gamma$, let $\hat{\tau}$ denote the connected subgraph of $\Gamma$ made up of the vertices and edges in $\tau$, along with all of the vertices and edges contained inside inner faces of $\tau$. Given any such $(\geq k)$-component $\tau$, we can form an $N$-map $\Gamma'$ by contracting all of $\hat{\tau}$ onto a single vertex $v$ at height $k$. Then $(\Gamma,\hat{\tau})$ is called an {\em upper expansion} of $(\Gamma',v)$, and $\hat{\tau}$ is called the {\em inserted component} of the upper expansion. Conversely, $(\Gamma',v)$ is called the {\em upper contraction} of $(\Gamma,\hat{\tau})$. We define {\em lower expansions} and {\em lower contractions} similarly.

\begin{figure}[ht]
\setlength{\captionindent}{0pt}
   \includegraphics[width=0.97\linewidth]{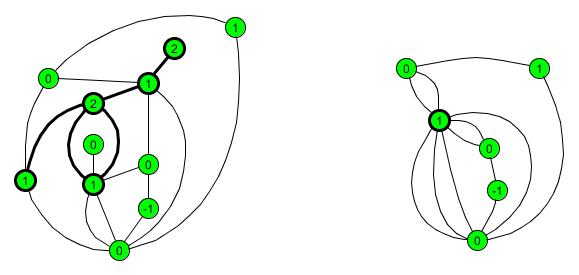} 
   \caption{On the left is an example of an $N$-map $\Gamma$ with an emphasised $(\geq1)$-component $\tau$. The upper contraction of $(\Gamma,\tau)$ is shown on the right.}
   \label{fig:Nmap_with_contraction}
\end{figure}

It will be convenient to enumerate $N$-maps in which the root edge joins the root vertex to a vertex at height 1. We will call such a map an $N^+$-map. Clearly for any $n\geq1$, exactly half of the $N$-maps with $n$ edges are $N^+$-maps. In an $N^+$-map, we will sometimes refer to the root vertex as the root-0 vertex, and the other vertex incident on the root edge as the root-1 vertex. In order to  enumerate these, we define a number of generating functions which we will relate to each other. We will start with the 4-valent case.

\subsection{Functional equations for the 4-valent case}
Let $K(x)$ be the generating function for $N^{+}$ maps, where each face has degree 4, counted by edges. Then the generating function $\tilde{A}(x)$ for 4-valent Eulerian orientations, counted by edges is given by $\tilde{A}(x)=1+2K(x)$. As the number of edges in a 4-valent orientation is exactly twice the number of vertices, the generating function $A(x)$ which counts these maps by vertices is given by the equation $A(x^2)=1+2K(x)$. The functions we use to calculate coefficients of $K$ will count the following generalisation of these $N^+$-maps. Define a $4^*$-map to be an $N^+$-map in which some vertices may be called {\em contracted}, and some corners may be highlighted, which satisfies the following properties:
\begin{itemize}
\item Each inner face has degree 2 or 4.
\item Every vertex around the outer face is at height 0 or 1.
\item In each inner face with degree 2, one of the two corners is highlighted. No other corners in the map are highlighted.
\item For any highlighted corner, the corresponding vertex must be contracted. 
\item All vertices adjacent to any given contracted vertex have the same height.
\end{itemize}
A contracted vertex is called {\em upper contracted} if the adjacent vertices are lower than it, and {\em lower contracted} otherwise.

\begin{figure}[ht]
\setlength{\captionindent}{0pt}
   \includegraphics[width=0.6\linewidth]{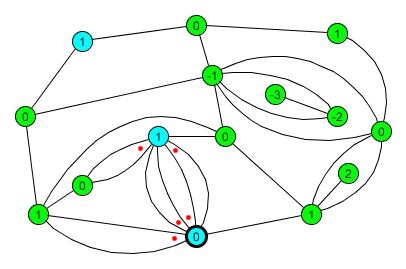} 
   \caption{An example of a $4^*$ map with the root vertex emphasised. The contracted vertices are coloured blue and the highlighted corners are shown by red dots.}
   \label{fig:4star-map}
\end{figure}

Let $\Gamma'$ be a $4^*$-map and let $v$ be an upper contracted vertex of $\Gamma'$. We will call an {\em upper expansion} $(\Gamma,\hat{\tau})$ of $(\Gamma',v)$ 4-valent if $\Gamma$ is a $4^*$-map and no vertices of $\hat{\tau}$ are contracted. 4-valent {\em lower expansions} are defined similarly.

Now we are ready to define the functions which we will use to calculate $K(x)$:
\begin{itemize}
\item Let $J(x,c)$ be the generating function for $4^*$-maps, with no contracted vertices, where $x$ counts the edges, and $c$ counts the half-degree of the outer face. We also include the graph in which the root vertex is the only vertex. This graph contributes 1 to $J(x,c)$.
\item Let $G(x,b,c)$ be the generating function for $4^*$-maps with no contracted vertices. Here $x$ counts the edges, $b$ counts the degree of the root-1 vertex $v_1$, and $c$ counts the half-degree of the outer face. We also include the graph with only 1 vertex, which contributes 1 to $G(x,b,c)$.
\item Let $P(x,a,b,c)$ be the generating function for $4^*$-maps, in which the root-0 vertex, $v_{0}$, is the only contracted vertex. Here $x$ counts the edges, $a$ counts the number of highlighted corners around $v_0$, $b$ counts the degree of the root-1 vertex $v_1$, and $c$ counts the half-degree of the outer face.
\item Finally let $\Lambda_{z}$ be the linear operator defined by $\Lambda_{z}(z^n)=[c^n]J(x,c)$.
\end{itemize}
Since there are only finitely many $4^*$-maps with any given number of vertices, each of these generating functions is a series in $x$ where each coefficient is a polynomial in the other variables. The first few terms of each series are as follows:
\begin{align*}
J(x,c)&=1+cx+2c^2x^2+(4c+5c^3)x^3+\ldots\\\\
G(x,b,c)&=1+cbx+(bc^2+b^2c^2)x^2+(2b^2c+2b^3c+2bc^3+2b^2c^3+b^3c^3)x^3+\ldots\\\\
P(x,a,b,c)&=bcx+(ab^2c+bc^2+b^2c^2)x^2\\
&+(a^2 b^3 c+a b^3 c^2+a b^2 c^2+a b c^2+b^3 c^3+2 b^3 c+2 b^2 c^3+b^2 c+2 b c^3)x^3+\ldots\end{align*}

Now we will prove that these series are characterised by the following system of equations:

\begin{align*}G(x,b,c)&=1+\Lambda_{z}(P(x,z,b,c)),\\\\
J(x,c)&=G(x,1,c),\\\\
P(x,a,b,c)&=x^2b^2\frac{P(x,a,b,c)-P(x,a,1,c)}{b-1}\\
&+xbP(x,a,b,c)(a+2[c^1]G(x,b,c))\\
&+xbc(1+P(x,a,1,c))G(x,b,c),\\\\
\Lambda_{z}\left(z^n\right)&=[c^{n}]J(x,c)\text{ for }n\geq0.\end{align*}
Moreover, we will show that the generating function $K(x)$ is given by the equation
\[K(x)=\frac{1}{x}[c^1]J(x,c).\]

\begin{figure}~
\setlength{\captionindent}{0pt}
   \put(-207,0){\includegraphics[width=0.97\linewidth]{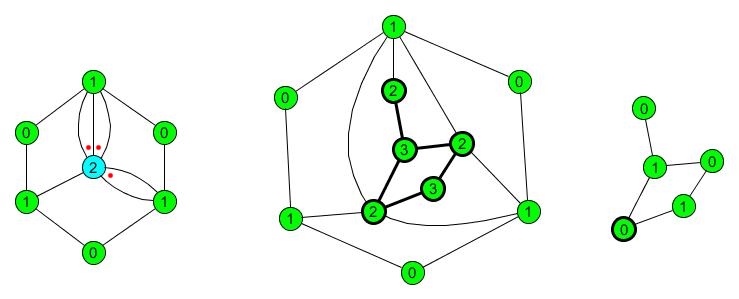}}
   \put(-200,24){\Large$\Gamma'$}
   \put(-30,14){\Large$\Gamma$}
   \put(-178,68){$e$}
   \put(-161,59){$v$}
   \put(-23,50){$e$}
   \caption{On the left is an example of a $4^{*}$-map $\Gamma'$ as in Proposition \ref{Lambdaprop}. The upper contracted vertex $v$ of lambda has height 2 and is surrounded by 3 highlighted corners. The map in the centre is a possible upper expansion $\Gamma$ of $(\Gamma',v)$ with the inserted component $\hat{\tau}$ emphasized. On the right is the corresponding $4^*$-map which is counted by $J(x,c)$, with its root vertex emphasised.}
   \label{fig:Lambda_proof}
\end{figure}
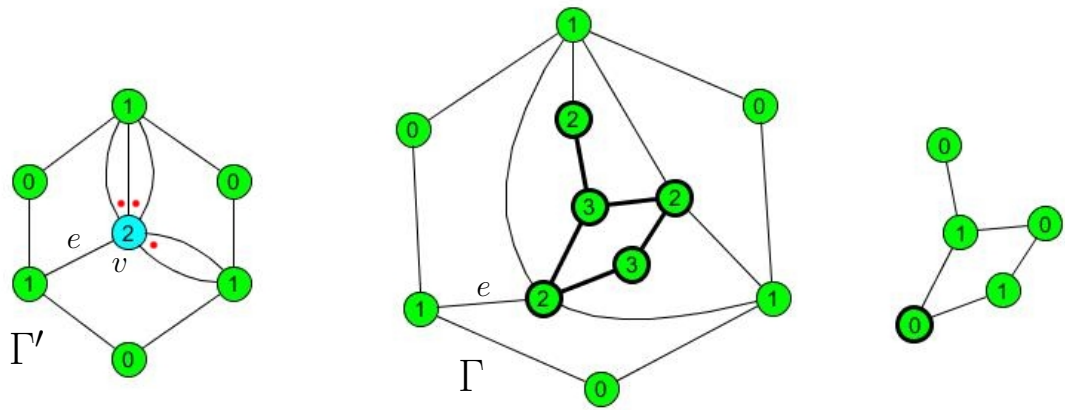
\begin{prop}\label{Lambdaprop} Let $\Gamma'$ be a $4^*$-map with an upper contracted vertex $v$ at height $k$ and let $n$ be the number of highlighted corners around $v$. Then the generating function $M_{v}(x)$ for 4-valent upper expansions $(\Gamma,\hat{\tau})$ of $(\Gamma',v)$, counted by edges in $\hat{\tau}$ is given by $M_{v}(x)=\Lambda_{z}(z^n)=[c^n]J(x,c)$.\end{prop}
\begin{proof} For any non-negative integer $m$, the coefficient $[x^m][c^n]J(x,c)$ is equal to the number of $4^*$-maps, with no contracted vertices, which contain $m$ edges and where $n$ is the half-degree of the outer face. We just need to prove that these are in bijection with 4-valent upper expansions $(\Gamma,\hat{\tau})$, where $\hat{\tau}$ contains $m$ edges. Let $e$ be a fixed edge of $\Gamma'$ which is incident on $v$. For any 4-valent upper expansion $(\Gamma,\hat{\tau})$ of $(\Gamma',v)$, we will consider the vertex of $\hat{\tau}$ which is incident on $e$ to be the root vertex of $\hat{\tau}$. Now we will show that for any such upper expansion, the degree of the outer face of $\hat{\tau}$ is $2n$.

Let $F$ be any inner face of $\Gamma$. There are three possibilities: either $F$ is an inner face of $\hat{\tau}$, or $F$ contains no edges in common with $\hat{\tau}$ or $F$ contains two edges in common with each of $\hat{\tau}$ and $\Gamma\setminus\hat{\tau}$. In the first case, $F$ does not correspond to a face of $\Gamma'$. In the second case, $F$ corresponds to a face of $\Gamma'$ with the same degree. If $F$ has degree 2, the highlighted corner of $F$ involves the same vertex in $\Gamma$ and $\Gamma'$, in particular, since the vertices of $\hat{\tau}$ are not contracted, this vertex cannot be $v$. In the final case, $F$ must contain two (outer) edges of $\hat{\tau}$, so the degree of the outer face of $\hat{\tau}$ is equal to twice the number of these faces. Moreover, each of these faces corresponds to a face of degree 2 in $\Gamma'$, with a highlighted vertex at $v$. Hence, the number of these faces is equal to $n$, the number of highlighted corners around $v$. Therefore, the outer face of $\hat{\tau}$ has degree $2n$.

Since each outer edge of $\hat{\tau}$ is contained in one of these faces, this implies that each outer vertex of $\hat{\tau}$ is contained in one of these faces. Hence, the outer vertices of $\hat{\tau}$ must each be at height $k$ or $k+1$. Hence, if we subtract $k$ from the height of every vertex in $\hat{\tau}$ to form a new map, then this new map is a $4^*$-map with no contracted vertices. Moreover, this transformation is reversible, so we can take any $4^*$-map with no contracted vertices, with $m$ edges and an outer face of degree $2n$, and construct a corresponding upper expansion $(\Gamma,\hat{\tau})$ of $(\Gamma',v)$. This completes the proof that the two sets are in bijection, which implies that $M_{v}(x)=[c^n]J(x,c)=\Lambda_{z}(z^n)$.\end{proof}

Similarly to this Proposition, we obtain an equivalent result for lower expansions if $v$ is a lower contracted vertex.

Now we are ready to prove each of the equations.

\begin{prop}The generating function $G$ is given by the equation
\[G(x,b,c)=1+\Lambda_{z}(P(x,z,b,c))\]\end{prop}
\begin{proof} This result follows immediately from the fact that the non-atomic $4^{*}$ maps $\Gamma$ which are counted by $G$ are exactly the lower expansions around the root-0 vertex of the maps $\Gamma'$ which are counted by $P$. The atomic map contributes 1 to $G(x,b,c)$.
\end{proof}

\begin{prop}The generating function $J$ is given by the equation
\[J(x,c)=G(x,1,c).\]
\end{prop}
\begin{proof}
By definition, $J(x,c)$ and $G(x,b,c)$ count the same maps, the only difference is that in $G$ there is a weight $b$ which counts the degree of the root-1 vertex, $v_1$. Hence $J(x,c)=G(x,1,c)$.
\end{proof}

\begin{figure}~
\setlength{\captionindent}{0pt}
   \put(-207,0){\includegraphics[width=0.97\linewidth]{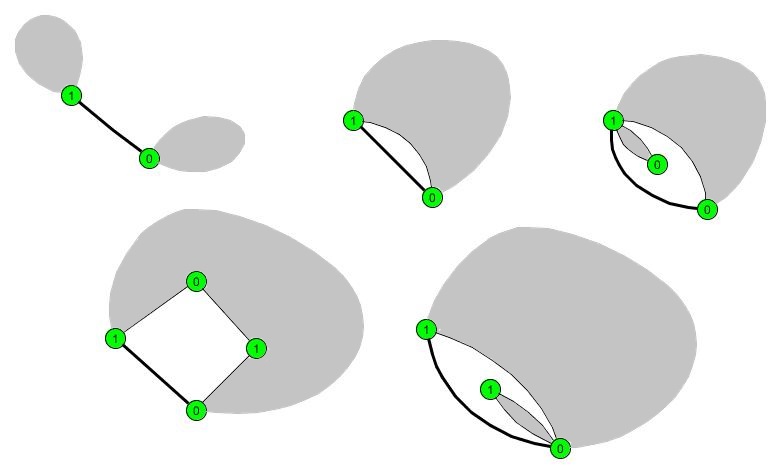}}
   \put(-157,173){$e$} 
   \put(-138,156){$v_{0}$} 
   \put(-180,189){$v_{1}$} 
   \put(-5,157){$e$} 
   \put(12,133){$v_{0}$}
   \put(-31,176){$v_{1}$} 
   \put(126,141){$e$} 
   \put(157,128){$v_{0}$} 
   \put(79,175){$u_{1}=v_{1}$} 
   \put(143,151){$u$} 
   \put(-132,42){$e$} 
   \put(-113,22){$v_{0}$} 
   \put(-156,60){$v_{1}$}
   \put(-89,63){$u_{1}$}
   \put(-107,88){$u$}
   \put(32,23){$e$} 
   \put(56,0){$v_{0}=u$} 
   \put(43,33){$u_{1}$} 
   \put(9,64){$v_{1}$} 
   \caption{The five different types of graphs which contribute to $P(x,a,b,c)$. The bottom two types are considered in the same case. In Proposition \ref{eqnforP}, The contributions to $P$ from the four cases are shown to be $xbc(P(x,a,1,c)+1)G(x,b,c)$, $abxP(x,a,b,c)$, $2xbP(x,a,b,c)[c^1]G(x,b,c)$ and $x^2b^2(P(x,a,b,c)-P(x,a,1,c))/(b-1)$, respectively.}
   \label{fig:Cases_for_P}
\end{figure}
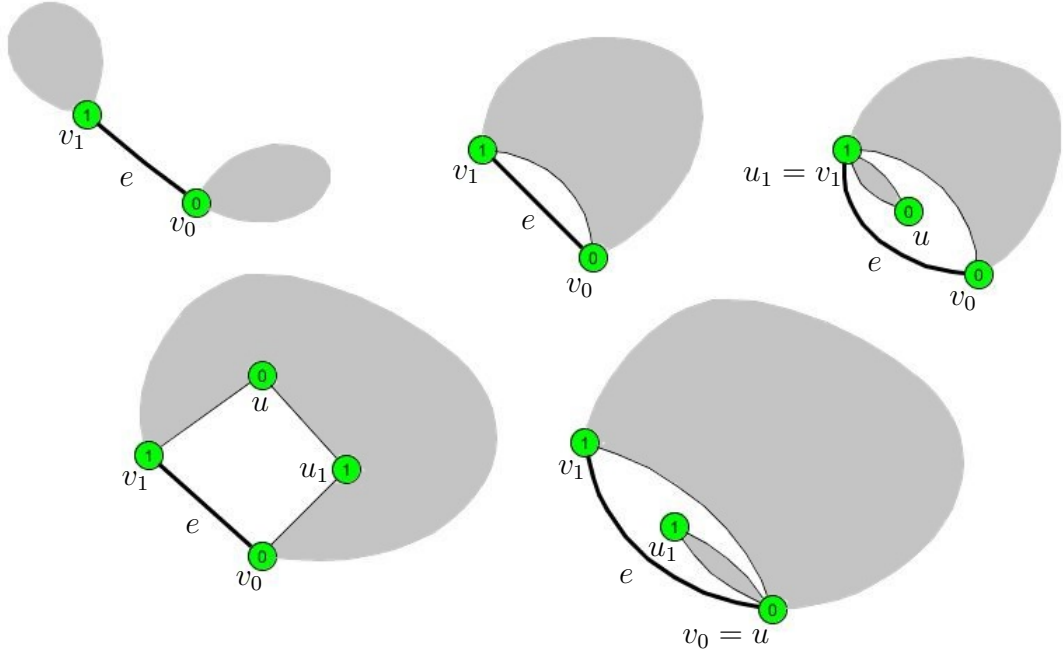

\begin{prop}\label{eqnforP}The generating function $P$ is given by the equation
\begin{align*}P(x,a,b,c)&=x^2b^2\frac{P(x,a,b,c)-P(x,a,1,c)}{b-1}\\
&+xbP(x,a,b,c)(a+2[c^1]G(x,b,c))\\
&+xbc(1+P(x,a,1,c))G(x,b,c)\end{align*}
\end{prop}
\begin{proof}
Let $\Gamma$ be a graph which is counted by $P$, and let vertices $v_{0}$, $v_{1}$ and edge $e$ be the root-0 vertex, root-1 vertices and root edge of $\Gamma$, respectively. First we will consider the case where removing $e$ disconnects the graph. In this case, let $\Gamma_{0}$ be the component containing $v_{0}$, and let $\Gamma_{1}$ be the component containing $v_{1}$. Since $\Gamma_{0}$ can be any $4^*$-map, where $v_{0}$ is the only contracted vertex, the possibilities for $\Gamma_{0}$ are counted by $P(x,a,1,c)+1$. The $+1$ comes from the fact that $\Gamma_{0}$ may be the atomic map. Since $\Gamma_{1}$ has no contracted vertices, the possibilities for it are counted by $G(x,b,c)$. The edge $e$ obviously contributes one edge, increases the half degree of the outer face by 1 and contributes 1 to the degree of $v_{1}$. Hence, this case contributes $xbc(P(x,a,1,c)+1)G(x,b,c)$ to the generating function $P$.

Now we will consider the case where removal of $e$ does not disconnect the graph. Then, since the face immediately anticlockwise from $e$ around $v_{0}$ is the outer face, the face on the opposite side of $e$ must be an inner face. First we will consider the case where this face has degree 2. In this case, removing $e$ forms another graph $\Gamma'$ which is counted by $P$. Since adding $e$ adds 1 to the number of edges, the number of faces with degree 2 around $v_{0}$ and the degree of $v_{1}$, the contribution from this case is $abxP(x,a,b,c)$. 

In the remaining cases, $e$ is adjacent to an inner face $F_{0}$ with degree 4. Let $v_{0},v_{1},u,u_{1}$ be the vertices around this face (in clockwise order). Now we will consider the case where $v_{1}=u_{1}$. Let $\Gamma_1$ be the map formed by the two edges of $F_{0}$ between $u$ and $v_{1}$ along with everything contained in inner faces formed by these edges. Let $\Gamma_{2}$ be the map formed from $\Gamma$ by removing $\Gamma_{1}$ and $e$. Then $\Gamma$ is uniquely determined by $\Gamma_{2}$ and $\Gamma_{1}$. Moreover, $\Gamma_{2}$ can be any map counted by $P$, so the possible maps $\Gamma_{2}$ are counted by $P(x,a,b,c)$. If $u$ is labelled $0$, then $\Gamma_{1}$ can be any $4^*$-map with no contracted vertices, with outer degree 2. If $u$ is labelled 2, then replacing every label $t$ in $\Gamma_{1}$ with $2-t$ yields any $4^*$-map with no contracted vertices, with outer degree 2. Hence, the possibilities for $\Gamma_{1}$ are counted by $2[c^1]G(x,b,c)$. Finally, the outer degree of $\Gamma$ is equal to the outer degree of $\Gamma_{2}$, and the edge $e$ contributes 1 to both the number of edges and the degree of $v_{1}$. Hence, the contribution from this case is $2xbP(x,a,b,c)[c^1]G(x,b,c)$.

Finally we are left with the case where the inner face $F_{0}$ has vertices $v_{0},v_{1},u,u_{1}$ with $v_{1}\neq u_{1}$. Let $\Gamma_{1}$ be the map formed from $\Gamma$ by identifying $u_{1}$ with $v_{1}$, then removing $e$ and one of the edges between $v_{1}$ and $u$ which borders the face with degree 2 formed between $u$ and $v_{1}$. Then $\Gamma_{1}$ can be any map counted by the generating function $P$. To reverse this procedure, we must duplicate an edge adjacent to $v_{1}$ in $\Gamma_{1}$ and also duplicate the root edge, then split the vertex $v_{1}$ into two vertices in such a way that the two faces with degree 2 join to make a quadrangle. Assume that $\Gamma_{1}$ contributes $x^{n}a^{m}b^{k}c^{l}$ to $P(x,a,b,c)$, then $\Gamma_{1}$ has $n$ edges, $m$ faces with degree 2, the outer face has degree $2l$, and $v_{1}$ has degree $k$. We will now calculate the contribution to $P(x,a,b,c)$ of all possible $4^*$-maps $\Gamma$ corresponding to this map. 

There are $k$ possible choices for the edge incident on $v_{1}$ to duplicate so as to form $\Gamma$, and for each choice, the $k+1$ resulting edges are split between $v_{1}$ and $u_{1}$. For each choice, the resulting degree of $v_{1}$ is a distinct number between 2 and $k+1$. Hence, the possible graphs $\Gamma$ are counted by
\[x^{n+2}a^{m}c^{l}(b^{k+1}+b^{k}+\ldots+b^2)=b^2x^{n+2}a^{m}c^{l}\frac{b^{k}-1}{b-1},\]
We get the total contribution from this case by summing the above expression over all maps $\Gamma_{1}$ counted by $P$, which gives
\[x^2b^2\frac{P(x,a,b,c)-P(x,a,1,c)}{b-1}.\]
Adding the contributions from each of the four cases gives the desired result.
\end{proof}

\begin{prop}
The generating function $K(x)$ for $N^{+}$ maps, where each face has degree 4 is given by
\[K(x)=\frac{1}{x}[c^1]J(x,c).\]
\end{prop}
\begin{proof}
The expression $[c^1]J(x,c)$ counts $4^*$-maps with outer degree 2. If we remove the edge on the outer face which is not the root edge from such a map, we get an $N^{+}$ map, where each face has degree 4. Moreover, this procedure is clearly reversible. Hence, since the procedure removes one edge, we get the desired equation.
\end{proof}

\subsection{Functional equations for the general case}

In this section we find a system of functional equations which allow us to enumerate $N^+$-maps by edges in polynomial time. Let $V(x)$ be the generating function for $N^{+}$-maps, counted by edges. Then the generating function $U(x)$ for rooted planar Eulerian orientations counted by edges is given by $U(x)=2V(x)+1$. Define an $N^{*}$ map to be an $N^{+}$ map in which some vertices may be called {\em contracted} such that all corners of the outer face have non-negative height and all vertices adjacent to a given contracted vertex must be at the same height. A contracted vertex is called {\em upper contracted} if the adjacent vertices are lower than it, and {\em lower contracted} otherwise. Finally, we define the \textit{inner degree} of a vertex in an $N^{*}$-map to be the number of corners around that vertex which are not corners of the outer face. Now we will define the other functions:
\begin{itemize}
\item Let $F(x,c)$ be the generating function for $N^{*}$-maps with no contracted vertices, where $c$ counts the number of corners of the outer face at height 0 and $x$ counts the edges. In this count we also include the map in which the root vertex is the only vertex. This contributes 1 to $F(x,c)$.
\item Let $R(x,a,b)$ be the generating function for $N^{*}$-maps, where the outer face has degree 2, in which the only contracted vertices are the root-0 vertex, $v_0$, and the root-1 vertex, $v_{1}$, where $x$ counts the edges, $a$ counts the degree of $v_{0}$ and $b$ counts the degree of $v_{1}$.
\item Let $S(x,a,b)$ be the generating function for $N^{*}$-maps, where the outer face has degree 2, in which the only contracted vertices are the root-0 vertex, $v_0$, and the root-1 vertex, $v_{1}$, and there are exactly two edges between $v_{0}$ and $v_{1}$, where $x$ counts the edges, $a$ counts the degree of $v_{0}$ and $b$ counts the degree of $v_{1}$.
\item Let $H(x,b,c)$ be the generating function for $N^{*}$-maps in which the root-1 vertex, $v_{1}$, is the only contracted vertex, where $x$ counts the edges, $b$ counts the degree of $v_{1}$ and $c$ counts the number of corners of the outer face at height 0. In this count, we also include the map in which the root-1 vertex is the only vertex. This contributes 1 to $H(x,b,c)$.
\item Let $M(x,a,c)$ be the generating function for $N^{*}$-maps in which the root-0 vertex, $v_{0}$, is the only contracted vertex, where $x$ counts the edges, $a$ counts the inner degree of $v_{0}$ and $c$ counts the number of corners of the outer face at height 0. We also include the map in which the root vertex is the only vertex. This contributes 1 to $M(x,a,c)$.
\item Let $T(x,a,b,c)$ be the generating function for $N^{*}$-maps in which the root-0 vertex, $v_{0}$, and the root-1 vertex, $v_{1}$ are the only contracted vertices, and the root edge is the only edge between these vertices, where $x$ counts the edges, $a$ counts the inner degree of $v_{0}$, $b$ counts the degree of $v_{1}$ and $c$ counts the number of corners of the outer face at height 0.
\item Finally let $\Omega_{z}$ be the linear operator defined by $\Omega_{z}(z^0)=1$ and
\[\Omega_{z}(z^n)=\sum_{j=0}^{\infty}{n+j-1\choose n-1}[c^j]F(x,c),\]
for $n>0$.
\end{itemize}
Since there are only finitely many $N^+$-maps with any given edges, each of these generating functions is a series in $x$ where each coefficient is a polynomial in the other variables. The first few terms of each series are as follows:
\begin{align*}
R(x,a,b)&=xab+x^2 a^2 b^2 +x^3 \left(a^3 b^3+a^3 b^2+a^2 b^3\right)+\ldots\\\\
S(x,a,b)&= x^2a^2b^2\begin{aligned}[t] &+x^3 (a^3 b^2+a^2b^3)+x^4 \left(2 a^4 b^2+a^3 b^3+2 a^3 b^2+2a^2 b^4+2a^2 b^3\right)\\
&+x^5 \left(5 a^5 b^2+2 a^4 b^3+8 a^4 b^2+2 a^3 b^4+5 a^3 b^3+10 a^3 b^2+5a^2 b^5+8a^2 b^4+10a^2 b^3\right)+\ldots\end{aligned}\\\\
F(x,c)&=1+c x+2 \left(c^2+c\right) x^2+\left(5 c^3+8 c^2+10 c\right) x^3+\ldots\\\\
H(x,b,c)&=1\begin{aligned}[t]+b c x&+x^2 \left(b^2 c^2+b^2 c+b c^2\right)\\
&+x^3 \left(b^3 c^3+2 b^3 c^2+2 b^3 c+2 b^2 c^3+b^2 c^2+2 b^2 c+2 b c^3+2 b c^2\right)+\ldots\end{aligned}\\\\
T(x,a,b,c)&=xbc+x^2 (b^2 c^2+bc^2)+x^3\left(ab c^2+b^3 c^3+b^3 c^2+2 b^2 c^3+2b c^3+bc^2\right)+\ldots
\end{align*}

Now we will show that these are characterised by the following system of equations:
\begin{align*}R(x,a,b)&=abx+\frac{1}{abx}R(x,a,b)S(x,a,b),\\\\
S(x,a,b)&=\Omega_{z}\left(x^2a^2b^2+\frac{zS(x,a,b)-bS(x,a,z)}{z(b-z)}R(x,a,z)b+\frac{a^2}{z^2}R(x,z,b)S(x,z,b)\right),\\\\
H(x,b,c)&=\Omega_{z}\left(\frac{1}{xbz}T(x,z,b,c)R(x,z,b)\right)+1,\\\\
M(x,a,c)&=\Omega_{z}\left(\frac{1}{xaz}T(x,a,z,c)R(x,a,z)\right)+1,\\\\
F(x,c)&=\Omega_{z}(H(x,z,c)),\\\\
\Omega_{z}(z^{0})&=1,\\\\
\Omega_{z}(z^{n})&=\sum_{j=0}^{\infty}{n+j-1\choose n-1}[c^j]F(x,c)\text{ for }n>0,\\\\
T(x,a,b,c)&=\Omega_{z}\left(\frac{T(x,a,b,c)-T(x,a,z,c)}{b-z}R(x,a,z)b\right)
+xb(c-a)H(x,b,c)M(x,a,c)+xabH(x,b,c).\end{align*}

Moreover, \[V(x)=\Omega_{y}\left(\Omega_{z}\left(\frac{1}{x^2y^2z^2}R(x,y,z)S(x,y,z)\right)\right).\]

\begin{prop}Let $\Gamma'$ be an $N^*$-map with a lower contracted vertex $v$ at height $0$ and let $n$ be the inner degree of $v$. Then the generating function $M_{v}(x)$ for lower expansions $(\Gamma,\hat{\tau})$ of $(\Gamma',v)$, where $\Gamma$ is an $N^{*}$-map, counted by edges in $\hat{\tau}$ is given by $M_{v}(x)=\Omega_{z}(z^n)$.\end{prop}
\begin{proof}
Since $(\Gamma,\hat{\tau})$ is a lower expansion of $(\Gamma',v)$, and $v$ is a vertex at height 0, all outer vertices of $\hat{\tau}$ are contained in some $(\leq0)$-component $\tau$, so these vertices must have non-positive height. If $n=0$, then the inner degree of $v$ is 0, so all outer vertices of $\hat{\tau}$ are also outer vertices of $\Gamma$. But since $\Gamma$ is an $N^{*}$-map, all of its outer vertices have non-negative height. Hence the outer vertices of $\hat{\tau}$ must all have height 0, which is only possible if $\hat{\tau}$ is the graph with only one vertex. Hence in this case $M_{v}(x)=1=\Omega_{z}(z^0)$.

Now we will consider the case when $n\geq1$. First, highlight one of the edges in $\Gamma'$ which is incident on $v$. Since the outer vertices of $\hat{\tau}$ all have non-positive heights, we can obtain an $N^{*}$-map $\tau'$ (without contracted vertices) by changing each height $s$ in $\hat{\tau}$ to $-s$, using the convention that the root vertex $v_{0}$ of $\hat{\tau}$ is the vertex adjacent to the image of the highlighted edge in $\Gamma$. Recall that these $N^{*}$-maps are enumerated by $F(x,c)$. Consider a specific $N^*$-map $\tau'$, which contributes $x^kc^j$ to $F(x,c)$. Then the corresponding map $\hat{\tau}$ contains $k$ edges and around the outer face there are $j$ corners at height 0. We will calculate the contribution of this map $\hat{\tau}$ to $M_{v}(x)$. Clearly any specific lower expansion $(\Gamma,\hat{\tau})$ contributes $x^k$ to $M_{v}(x)$, so we just need to calculate the number of lower expansions $(\Gamma,\hat{\tau})$ of $(\Gamma',v)$. Going clockwise around the outer face of $\hat{\tau}$, starting at $v_{0}$, let $p_{1},p_{2},\ldots,p_{j}$ be the paths between vertices at height 0, so these partition the boundary of $\hat{\tau}$. Now let $c_{1},c_{2},\ldots,c_{n}$ be the inner corners around $v$ in $\Gamma'$, in clockwise order starting from the highlighted edge. Then in the lower expansion, each inner corner $c_{i}$ expands to contain a number $a_{i}$ of the paths $p_{1},\ldots,p_{j}$. Since the vertices in $\hat{\tau}$ which are not at height 0 have negative heights, each path $p_{t}$ must not be on the outside of $\Gamma$, so $p_{t}$ must be counted by one of the terms $a_{i}$. Moreover, due to the clockwise order, the lower expansion is uniquely determined by the sequence $a_{1},\ldots,a_{n}$, the only restrictions on this sequence being that each term $a_{i}$ is a non-negative integer and the sum of the terms is $j$. The number of such sequences is
\[{n+j-1\choose n-1}.\]
Hence the contribution of the $N^*$-map $\tau'$ to $M_{v}(x)$ is
\[{n+j-1\choose n-1}x^k.\]
Summing this over all $N^*$-maps gives the desired result:
\[M_{v}(x)=\sum_{j=0}^{\infty}{n+j-1\choose n-1}[c^{j}]F(x,c)=\Omega_{z}(z^n).\]
\end{proof}

\begin{prop}Let $\Gamma'$ be an $N^*$-map with an upper contracted vertex $v$ at height $1$ and let $n$ be the degree of $v$. Then the generating function $M_{v}(x)$ for upper expansions $(\Gamma,\hat{\tau})$ of $(\Gamma',v)$, counted by edges in $\hat{\tau}$ is given by $M_{v}(x)=\Omega_{z}(z^n)$.\end{prop}

\begin{proof}The proof is identical to the one above except that $c_{1},\ldots,c_{n}$ is the list of all corners around $v$, and the $N^{*}$-map $\tau'$ is constructed by subtracting 1 from each height in $\hat{\tau}$.\end{proof}

\begin{prop} The generating function $R$ is given by the equation
\[R(x,a,b)=abx+\frac{1}{xab}R(x,a,b)S(x,a,b).\]
\end{prop}
\begin{proof}Let $\Gamma$ be an $N^{*}$-map which is counted by $R$. Let $v_{0}$, $v_{1}$ and $e$ be the root-0 vertex, root-1 vertex and root edge, respectively. Clearly the case where $e$ is the only edge contributes $abx$ to $R$. Otherwise, there must be at least two distinct edges between $v_{0}$ and $v_{1}$. Let $e_{1}$ be the next edge clockwise around $v_{0}$ which connects to $v_{1}$, and let $e'$ be the next edge anticlockwise around $v_{0}$ from $e$, so $e$ and $e'$ are the two edges which border the outer face of $\Gamma$. Note that $e_{1}$ and $e'$ may or may not be the same edge. Let $\Gamma_{1}$ be the map formed by $e$ and $e_{1}$ and everything contained in the cycle formed by these edges. Similarly let $\Gamma_{2}$ be the map formed by the edges $e_{1}$ and $e'$ and everything they contain. Then $\Gamma_{1}$ can be any map which is counted by $S$ and $\Gamma_{2}$ can be any map which is counted by $R$, Hence the maps $\Gamma$ are counted by the product $R(x,a,b)S(x,a,b)$. However, the edge $e_{1}$ is counted twice in the product in $a$, $b$ and $x$. Hence the contribution from this case is 
\[\frac{1}{xab}R(x,a,b)S(x,a,b).\]
Adding the contribution from both cases gives the desired result.\end{proof}

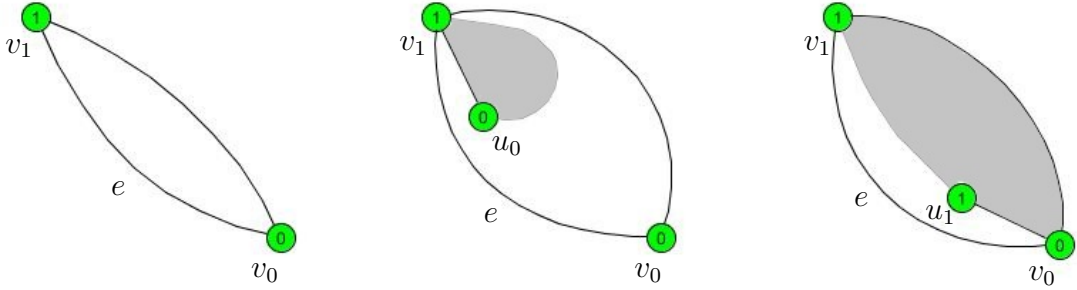
\begin{figure}~
\setlength{\captionindent}{0pt}
   \put(-207,0){\includegraphics[width=0.97\linewidth]{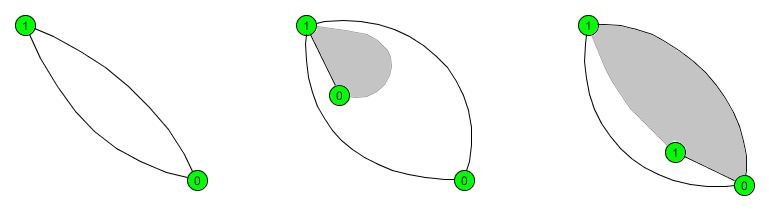}}
   \put(-165,32){$e$}
   \put(-112,1){$v_{0}$} 
   \put(-205,86){$v_{1}$} 
   \put(-24,22){$e$} 
   \put(-21,49){$u_{0}$} 
   \put(33,1){$v_{0}$} 
   \put(-56,87){$v_{1}$} 
   \put(116,28){$e$}
   \put(143,23){$u_{1}$}
   \put(182,0){$v_{0}$} 
   \put(97,87){$v_{1}$} 
   \caption{The three different cases of graphs which contribute to $S(x,a,b)$. The contributions to $S$ from the three cases are shown to be $x^2a^2b^2$, $\Omega_{z}\left(\frac{a^2}{z^2}R(x,z,b)S(x,z,b)\right)$ and $\Omega_{z}\left(R(x,a,z)b(zS(x,a,b)-bS(x,a,z))/(bz-z^2))\right)$, respectively.}
   \label{fig:Cases_for_S}
\end{figure}

\begin{prop}\label{eqnforS} The generating function $S$ is given by the equation
\[S(x,a,b)=\Omega_{z}\left(x^2a^2b^2+\frac{zS(x,a,b)-bS(x,a,z)}{z(b-z)}R(x,a,z)b+\frac{a^2}{z^2}R(x,z,b)S(x,z,b)\right).\]
\end{prop}
\begin{proof}Let $\Gamma$ be an $N^{*}$ map which is counted by $S$. Let $v_{0}$, $v_{1}$ and $e$ be the root-0 vertex, root-1 vertex and root edge of $\Gamma$. Let $e'$ be the other edge between $v_{0}$ and $v_{1}$. In the case where vertices $v_{0}$ and $v_{1}$ both have degree 2, the edges $e$ and $e'$ must be the only edges in the graph. Hence, this case contributes $x^2a^2b^2$ to $S(x,a,b)$. Next we will consider the case where $v_{0}$ has degree 2 but $v_{1}$ has degree greater than 2. Let $e_{1}$ be the next edge anticlockwise from $e$ around $v_{1}$, and let $u_{0}$ be the other vertex on edge $e_{1}$. Let $\tau$ be the $(\leq 0)$-component containing $u_{0}$ and let $(\Gamma',u_{0})$ be the lower contraction of $(\Gamma,\hat{\tau})$. Finally, let $\Gamma_{R}$ be the map formed from $\Gamma'$ by removing $v_{0}$ and the two edges attached to it, and adding a new edge $e_{2}$ between $u_{0}$ and $v_{1}$ so that $e_{1}$ and $e_{2}$ are the only edges on the outer face of $\Gamma_{R}$. Since $u_{0}$ and $v_{1}$ are contracted vertices in $\Gamma_{R}$, and the outer face has degree $2$, $\Gamma_{R}$ is counted by $R$. Since there are at least two edges between $u_{0}$ and $v_{1}$, this map cannot contain only a single edge. However, for any other map $\Gamma_{R}$ counted by $R$, the transformations between $\Gamma$ and $\Gamma_{R}$ can be reversed, so $\Gamma_{R}$ can be any other map counted by $R$. Hence, the possible maps $\Gamma_{R}$ are counted by
\[R(x,z,b)-xzb,\]
where $x$ counts the edges, $z$ counts the degree of $u_{0}$ and $b$ counts the degree of $v_{1}$. Since the transformation from $\Gamma_{R}$ to $\Gamma'$ just removes one edge between $u_{0}$ and $v_{1}$, and adds two between $v_{0}$ and $v_{1}$, the possible maps $\Gamma'$ are counted by
\[\frac{xa^2b}{z}(R(x,z,b)-xzb)=\frac{xa^2b}{z}R(x,z,b)-x^2a^2b^2.\]
Since $(\Gamma,\hat{\tau_{1}})$ can be any lower expansion of $(\Gamma',u_{0})$, the contribution to $S(x,a,b)$ from this case is
\[\Omega_{z}\left(\frac{xa^2b}{z}R(x,z,b)-x^2a^2b^2\right).\]
Using the previous Proposition, we can rewrite this as
\[\Omega_{z}\left(\frac{a^2}{z^2}R(x,z,b)S(x,z,b)\right).\]

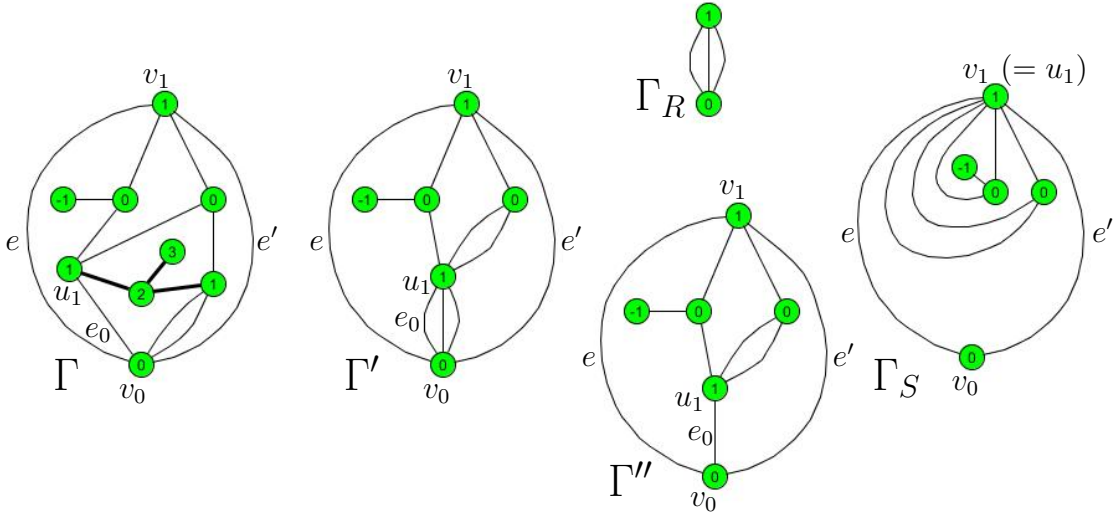
\begin{figure}~
\setlength{\captionindent}{0pt}
   \put(-207,0){\includegraphics[width=0.97\linewidth]{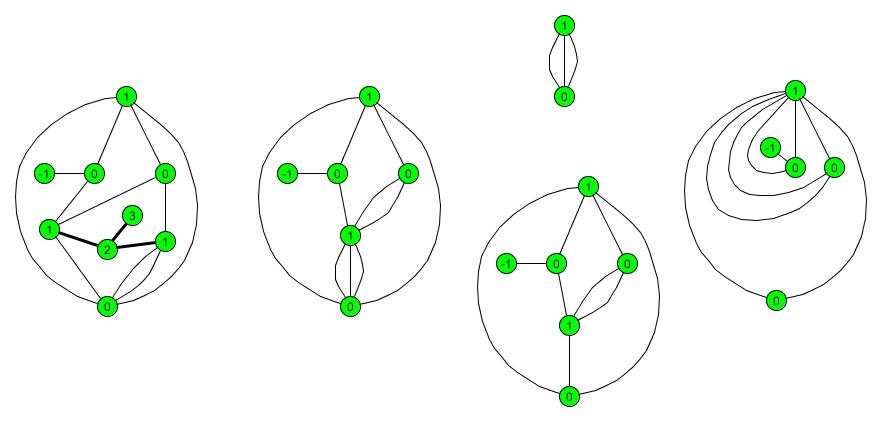}}
   \put(-190,44){\Large$\Gamma$}
   \put(-165,41){$v_{0}$}
   \put(-190,78){$u_{1}$}
   \put(-178,64){$e_{0}$}
   \put(-208,97){$e$}
   \put(-113,97){$e'$}
   \put(-156,160){$v_{1}$}
   \put(-80,44){\Large$\Gamma'$}
   \put(-50,41){$v_{0}$}
   \put(-58,82){$u_{1}$}
   \put(-61,68){$e_{0}$}
   \put(-93,97){$e$}
   \put(2,97){$e'$}
   \put(-41,160){$v_{1}$}
   \put(20,4){\Large$\Gamma''$}
   \put(51,0){$v_{0}$}
   \put(45,38){$u_{1}$}
   \put(50,25){$e_{0}$}
   \put(10,53){$e$}
   \put(105,53){$e'$}
   \put(62,118){$v_{1}$}
   \put(30,150){\Large$\Gamma_{R}$}
   \put(120,44){\Large$\Gamma_{S}$}
   \put(150,44){$v_{0}$}
   \put(108,100){$e$}
   \put(203,100){$e'$}
   \put(154,162){$v_{1}~(=u_{1})$}
   \caption{On the left is an example of a graph $\Gamma$ counted in the third case of Proposition \ref{eqnforS}, with the $(\geq 1)$-component $\tau$ highlighted. The other graphs shown are $\Gamma'$, $\Gamma''$, $\Gamma_{R}$ and $\Gamma_{S}$, which are involved in the decomposition of $\Gamma$. The contracted vertices are coloured blue and all other vertices are coloured green.}
   \label{fig:Last_case_in_S_proof}
\end{figure}
Finally, we will consider the case where $v_{0}$ has degree greater than 2. Let $e_{0}$ be the next edge clockwise from $e$ around $v_{0}$ and let $u_{1}$ be the other vertex connected to $e_{0}$. Let $\tau$ be the $(\geq1)$-component containing $u_{1}$, and let $(\Gamma',u_{1})$ be the upper contraction of $(\Gamma,\hat{\tau})$. Let $\Gamma_{R}$ be the map formed by all edges between $v_{0}$ and $u_{1}$ in $\Gamma'$ along with everything contained in the inner faces of these edges. Let $\Gamma''$ be the map formed from $\Gamma'$ by replacing all of $\Gamma_{R}$ with a single edge. Now let $\Gamma_{S}$ be the graph formed from $\Gamma''$ by deleting the edge $e_{0}$ and identifying $u_{1}$ with $v_{1}$ (this vertex in $\Gamma_{S}$ will be called $v_{1}$). In $\Gamma_{S}$, the edges $e$ and $e'$ still form the outer face and they are the only two edges between $v_{0}$ and $v_{1}$. Hence, $\Gamma_{S}$ is counted by the generating function $S$. Assume that $\Gamma_{S}$ contributes $x^na^md^k$ to $S(x,a,d)$. So $\Gamma_{S}$ has $n$ edges, and the degress of $v_{0}$ and $v_{1}$ in $\Gamma_{S}$ are $m$ and $k$, respectively. By analysing the transformation from $\Gamma''$ to $\Gamma_{S}$, we can see that in $\Gamma''$, the sum of the degrees of $u_{1}$ and $v_{1}$ is $k+1$, the number of edges is $n+1$ and the degree of $v_{0}$ is $m+1$. The degree of $v_{1}$ must be at least 2, and the degree of $u_{1}$ must be at least 1, but subject to these restrictions, there is exactly one map $\Gamma''$ for each choice of degrees of $v_{1}$ and $u_{1}$. Hence, the possible graphs $\Gamma''$ are counted by
\[x^{n+1}a^{m+1}(b^2z^{k-1}+b^3z^{k-2}+\ldots+b^kz)=x^{n+1}a^{m+1}\frac{b^{k+1}z-b^2z^k}{b-z},\]
where $b$ counts the degree of $v_{1}$ and $z$ counts the degree of $u_{1}$. Now, since the possible maps $\Gamma_{R}$ are counted by $R(x,a,z)$, and $\Gamma'$ is formed by combining any map $\Gamma''$ with any map $\Gamma_{R}$, while removing one edge between $v_{0}$ and $u_{1}$, the possible maps $\Gamma'$ are counted by
\[\frac{1}{xaz}x^{n+1}a^{m+1}\frac{b^{k+1}z-b^2z^k}{b-z}R(x,a,z)=x^{n}a^{m}\frac{b^{k+1}-b^2z^{k-1}}{b-z}R(x,a,z).\]
Then, since $(\Gamma,\hat{\tau})$ can be any upper expansion of $(\Gamma',u_{1})$, The possible graphs $\Gamma$ are counted by
\[\Omega_{z}\left(x^{n}a^{m}\frac{b^{k+1}-b^2z^{k-1}}{b-z}R(x,a,z)\right).\]
Summing this over all possible graphs $\Gamma_{S}$ gives the contribution from this case
\[\Omega_{z}\left(\frac{bS(x,a,b)-\frac{b^2}{z}S(x,a,z)}{b-z}R(x,a,z)\right).\]
Finally, adding the contributions from all three cases gives the desired result.
\end{proof}

\begin{prop} The generating function $H$ is given by the equation
\[H(x,b,c)=\Omega_{z}\left(\frac{1}{xbz}T(x,z,b,c)R(x,z,b)\right)+1.\]
\end{prop}
\begin{proof}Let $\Gamma$ be an $N^*$-map which is counted by $H$ such that $\Gamma$ is not just a single vertex. Let $v_{0}$, $v_{1}$ and $e$ be the root-0 vertex, root-1 vertex and root edge of $\Gamma$, respectively. Let $\tau$ be the $(\leq0)$-component containing $v_{0}$, and let $(\Gamma',v_{0})$ be the lower contraction of $(\Gamma,\hat{\tau})$. Let $\Gamma_{R}$ be the map formed by all edges between $v_{0}$ and $v_{1}$ in $\Gamma'$ along with everything contained in inner faces of these edges. Then the outer face of $\Gamma_{R}$ has degree 2, so the possible maps $\Gamma_{R}$ are counted by $R$. Let $\Gamma_{T}$ be the map formed from $\Gamma'$ by replacing all of $\Gamma_{R}$ with a single edge. In $\Gamma_{T}$, there is only one edge between $v_{0}$ and $v_{1}$, so $\Gamma_{T}$ is counted by the generating function $T$. Assume that $\Gamma_{R}$ contributes $x^{n}z^mb^k$ to $R(x,z,b)$. Then the transformation from $\Gamma'$ to $\Gamma_{T}$ decreases the number of edges by $n-1$, the degree of $v_{1}$ by $k-1$ and the inner degree of $v_{0}$ by $m-1$. Hence, if we let $z$ count the inner degree of $v_{0}$ in $\Gamma'$, then the possible maps $\Gamma'$ are counted by
\[x^{n-1}z^{m-1}b^{k-1}T(x,z,b,c).\]
Then, since $\Gamma$ can be any lower expansion of $(\Gamma',v_{0})$, the possible maps $\Gamma$ are counted by
\[\Omega_{z}(x^{n-1}z^{m-1}b^{k-1}T(x,z,b,c)).\]
Summing over all possible maps $\Gamma_{R}$ gives the contribution
\[\Omega_{z}\left(\frac{1}{xbz}T(x,z,b,c)R(x,z,b)\right).\]
Finally, adding 1 for the case when $\Gamma$ is a single vertex gives the desired result.\end{proof}

\begin{prop} The generating function $M$ is given by the equation
\[M(x,a,c)=\Omega_{z}\left(\frac{1}{xaz}T(x,a,z,c)R(x,a,z)\right)+1.\]
\end{prop}
\begin{proof}Let $\Gamma$ be an $N^*$-map which is counted by $M$ such that $\Gamma$ is not just a single vertex. Let $v_{0}$, $v_{1}$ and $e$ be the root-0 vertex, root-1 vertex and root edge of $\Gamma$, respectively. Let $\tau$ be the $(\geq1)$-component containing $v_{1}$, and let $(\Gamma',v_{1})$ be the upper contraction of $(\Gamma,\hat{\tau})$. Let $\Gamma_{R}$ be the map formed by all edges between $v_{0}$ and $v_{1}$ in $\Gamma'$ along with everything contained in inner faces of these edges. Then the outer face of $\Gamma_{R}$ has degree 2, so the possible maps $\Gamma_{R}$ are counted by $R$. Let $\Gamma_{T}$ be the map formed from $\Gamma'$ by replacing all of $\Gamma_{R}$ with a single edge. In $\Gamma_{T}$, there is only one edge between $v_{0}$ and $v_{1}$, so $\Gamma_{T}$ is counted by the generating function $T$. Assume that $\Gamma_{R}$ contributes $x^{n}a^mz^k$ to $R(x,a,z)$. Then the transformation from $\Gamma'$ to $\Gamma_{T}$ decreases the number of edges by $n-1$, the degree of $v_{1}$ by $k-1$ and the inner degree of $v_{0}$ by $m-1$. Hence, if we let $z$ count the degree of $v_{1}$ in $\Gamma'$, then the possible maps $\Gamma'$ are counted by
\[x^{n-1}a^{m-1}z^{k-1}T(x,a,z,c).\]
Then, since $\Gamma$ can be any upper expansion of $(\Gamma',v_{1})$, the possible maps $\Gamma$ are counted by
\[\Omega_{z}(x^{n-1}a^{m-1}z^{k-1}T(x,a,z,c)).\]
Summing over all possible maps $\Gamma_{R}$ gives the contribution
\[\Omega_{z}\left(\frac{1}{xaz}T(x,a,z,c)R(x,a,z)\right).\]
Finally, adding 1 for the case when $\Gamma$ is a single vertex gives the desired result.\end{proof}

\begin{prop} The generating function $F$ is given by the equation
\[F(x,c)=\Omega_{z}(H(x,z,c)).\]
\end{prop}
\begin{proof}Let $\Gamma$ be an $N^{*}$ map which is counted by $H$, and let $v_{0}$ be the root-0 vertex of $\Gamma$. Let $(\Gamma_{F},\hat{\tau})$ be any lower expansion of $(\Gamma,v_{0})$. Then the possible maps $\Gamma_{F}$ are exactly those which are counted by $F$. It follows immediately that $F(x,c)=\Omega_{z}(H(x,z,c))$.\end{proof}

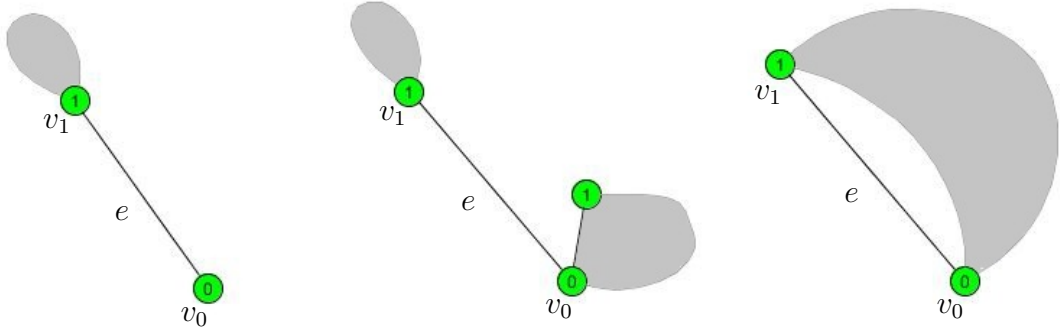
\begin{figure}~
\setlength{\captionindent}{0pt}
   \put(-207,0){\includegraphics[width=0.97\linewidth]{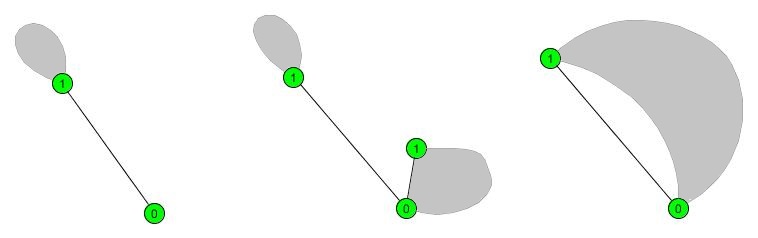}}
   \put(-158,39){$e$}
   \put(-133,1){$v_{0}$}
   \put(-185,74){$v_{1}$}
   \put(-27,43){$e$}
   \put(5,4){$v_{0}$}
   \put(-58,77){$v_{1}$}
   \put(118,46){$e$}
   \put(153,4){$v_{0}$}
   \put(84,85){$v_{1}$}
   \caption{The three different types of graph which contribute to $T(x,a,b,c)$. From left to right, the types are counted by $T_{0}(x,a,b,c)$, $T_{1}(x,a,b,c)$ and $T_{2}(x,a,b,c)$.}
   \label{fig:Cases_for_T}
\end{figure}

\begin{prop}\label{eqnforT} The generating function $T$ is given by the equation
\[T(x,a,b,c)=\Omega_{z}\left(\frac{T(x,a,b,c)-T(x,a,z,c)}{b-z}R(x,a,z)b\right)
+bx(c-a)H(x,b,c)M(x,a,c)+bxaH(x,b,c).\]
\end{prop}
\begin{proof}
Let $\Gamma$ be an $N^*$-map which is counted by $T$, and let $v_{0}$, $v_{1}$ and $e$ be the root-0 vertex, root-1 vertex and root edge of $\Gamma$, respectively. Let $T_{0}(x,a,b,c)$ be the contribution to $T$ from maps $\Gamma$ in which $v_{0}$ has degree 1. Let $T_{1}(x,a,b,c)$ be the contribution from maps in which $v_{0}$ has degree at least 2, but the removal of $e$ disconnects the graph. Let $T_{2}(x,a,b,c)$ be the contribution from maps in which the removal of $e$ does not disconnect the graph (which implies that $v_{0}$ has degree at least 2). Then \[T(x,a,b,c)=T_{0}(x,a,b,c)+T_{1}(x,a,b,c)+T_{2}(x,a,b,c).\]

First we will calculate $T_{0}$. Assume that $\Gamma$ is counted by $T_{0}$. Then if we remove $e$ and $v_{0}$, we get a map $\Gamma'$ counted by $H(x,b,c)$. Since the removal of $v_{0}$ and $e$ decreases the degree of $v_{1}$ by 1, the number of edges by 1 and the number of outer corners at height 0 by 1, we have the equation
\[T_{0}(x,a,b,c)=xbcH(x,b,c).\]

Now we will consider the case where removing $e$ disconnects the graph. Clearly, this case is enumerated by $T_{0}(x,a,b,c)+T_{1}(x,a,b,c)$.  In this case, let $\Gamma_{0}$ be the component containing $v_{0}$, and let $\Gamma_{1}$ be the component containing $v_{1}$. Since $\Gamma_{0}$ can be any $N^*$-map, where $v_{0}$ is the only contracted vertex, the possibilities for $\Gamma_{0}$ are counted by $R(x,a,c)$. Similarly, $\Gamma_{1}$ can be any $N^{*}$-map where $v_{1}$ is the only contracted vertex, so the possibilities for this are counted by $H(x,b,c)$. The edge $e$ obviously contributes one edge, and adds one to the degree of $v_{1}$, and also increases the number of outer corners at height 0 by 1 (since $v_{0}$ is on the outer face one further time). Hence, 
\[T_{0}(x,a,b,c)+T_{1}(x,a,b,c)=xbcM(x,a,c)H(x,b,c).\]
\begin{figure}~
\setlength{\captionindent}{0pt}
   \put(-207,0){\includegraphics[width=0.97\linewidth]{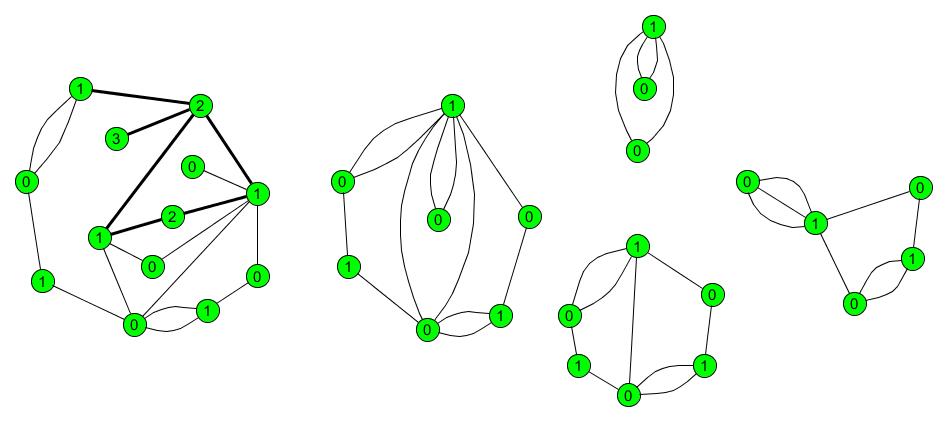}}
   \put(-200,25){\Large$\Gamma$}
   \put(-175,40){$e$}
   \put(-158,31){$v_{0}$}
   \put(-198,50){$v_{1}$}
   \put(-175,68){$u_{1}$}
   \put(-154,59){$e'$}
   \put(-70,35){\Large$\Gamma'$}
   \put(-45,45){$e$}
   \put(-30,30){$v_{0}$}
   \put(-64,57){$v_{1}$}
   \put(-24,145){$u_{1}$}
   \put(24,0){\Large$\Gamma''$}
   \put(52,10){$e$}
   \put(58,0){$v_{0}$}
   \put(32,17){$v_{1}$}
   \put(57,83){$u_{1}$}
   \put(35,140){\Large$\Gamma_{R}$}
   \put(130,60){\Large$\Gamma_{T}$}
   \put(156,40){$v_{0}$}
   \put(142,75){$t_{1}$}
   \caption{On the left is an example of a graph $\Gamma$ counted in the third case of Proposition \ref{eqnforT}, with the $(\geq 1)$-component $\tau$ highlighted. The other graphs shown are $\Gamma'$, $\Gamma''$, $\Gamma_{R}$ and $\Gamma_{T}$, which are involved in the decomposition of $\Gamma$.}
   \label{fig:Last_case_in_T_proof}
\end{figure}
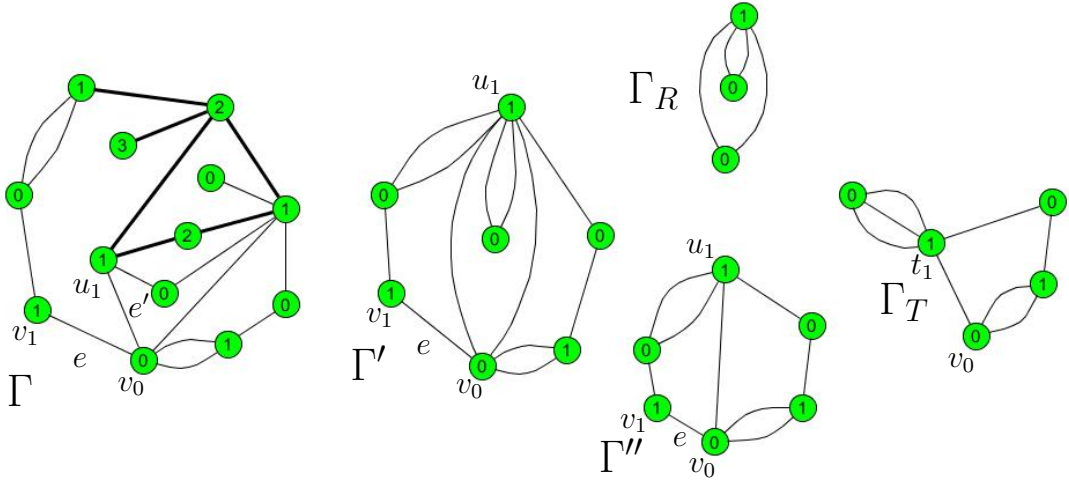

Now we will consider the case where $v_{0}$ has degree at least 2, however we will ignore the corner immediately clockwise from $e$ around $v_{0}$ in calculating the exponent of $c$ and $a$. So, this case is counted by $T_{1}(x,a,b,c)/c+T_{2}(x,a,b,c)/a$.  Let $e'$ be the next edge clockwise around $v_{0}$, and let $u_{1}$ be the other vertex connected to $e'$. Let $\tau$ be the $(\geq1)$-component containing $u_{1}$, and let $(\Gamma',u_{1})$ be the upper contraction of $(\Gamma,\hat{\tau})$. Now let $\Gamma_{R}$ be the map formed by all edges between $v_{0}$ and $u_{1}$ in $\Gamma'$ along with everything contained in the inner faces of these edges. Let $\Gamma''$ be the map formed from $\Gamma'$ by replacing all of $\Gamma_{R}$ with a single edge. Since $v_{0}$ is contracted, $u_{1}$ must have height 1. Now let $\Gamma_{T}$ be the graph formed from $\Gamma''$ by deleting the edge $e$ joining $v_0$ to $v_1$ and identifying $v_{1}$ and $u_{1}$ as the single vertex $t_{1}$. In $\Gamma_{T}$, $t_{1}$ is only adjacent to vertices at height 0 and $v_{0}$ is only adjacent to vertices at height 1, so $\Gamma_{T}$ is counted by the generating function $T$. Assume that $\Gamma_{T}$ contributes $x^na^mb^kc^l$ to $T(x,a,b,c)$. So $\Gamma_{T}$ has $n$ edges, there are $l$ outer corners at height 0, the vertex $t_{1}$ has degree $k$ and the vertex $v_{0}$ in $\Gamma''$ has inner degree $m$. In the transformation from $\Gamma''$ to $\Gamma_{T}$, no inner corners are removed, except perhaps the corner between $e$ and $e'$, which we don't count. Moreover, the sum of the degrees of $u_{1}$ and $v_{1}$ in $\Gamma''$ is $k+1$ and the number of edges in $\Gamma''$ is $n+1$. Note that the transformation from $\Gamma''$ to $\Gamma_{T}$ does not affect the number of $0$'s around the outer face, except perhaps at the corner which we don't count. Hence, the possible graphs $\Gamma''$ are counted by
\[x^{n+1}a^{m}c^{l}(b^{k}z+b^{k-1}z^2+\ldots+bz^k)=x^{n+1}a^{m}c^{l}\frac{bz(b^k-z^k)}{b-z},\]
where $b$ counts the degree of $v_{1}$ and $z$ counts the degree of $u_{1}$. Clearly the possible maps $\Gamma_R$ are counted by $R(x,a,z)$. Hence, the possible maps $\Gamma'$ are counted by
\[\frac{1}{xaz}x^{n+1}a^{m}c^{l}\frac{bz(b^k-z^k)}{b-z}R(a,z,c)=x^{n}a^{m-1}c^{l}\frac{b(b^k-z^k)}{b-z}R(a,z,c).\]
Since $\Gamma$ can be any upper expansion of $\Gamma'$ at $u_{1}$, the possible maps $\Gamma$ are counted by
\[\Omega_{z}\left(x^{n}a^{m-1}c^{l}\frac{b(b^k-z^k)}{b-z}R(a,z,c)\right).\]
Summing over all possible maps $\Gamma_{T}$ gives the contribution 
\[\Omega_{z}\left(\frac{b(T(x,a,b,c)-T(x,a,z,c))}{a(b-z)}R(a,z,c)\right)\]
from this case. Hence, 
\[\frac{1}{c}T_{1}(x,a,b,c)+\frac{1}{a}T_{2}(x,a,b,c)=\Omega_{z}\left(\frac{b(T(x,a,b,c)-T(x,a,z,c))}{a(b-z)}R(a,z,c)\right).\]
Finally, combining the four equations relating $T$, $T_{0}$, $T_{1}$ and $T_{2}$ gives the desired result.
\end{proof}

\begin{prop} The generating function $V$ for $N^{+}$-maps counted by edges is given by the equation
\[V(x)=\Omega_{y}\left(\Omega_{z}\left(\frac{1}{xyz}R(x,y,z)-1\right)\right).\]
\end{prop}
\begin{proof}
Let $\Gamma$ be an $N^{+}$-map, and let $v_{0}$, $v_{1}$ and $e$ be the root-0 vertex, root-1 vertex and root edge of $\Gamma$ respectively. Let $\tau_{0}$ be the $(\leq0)$-component containing $v_{0}$ and let $\tau_{1}$ be the $(\geq1)$-component containing $v_{1}$. Now let $(\Gamma',v_{0})$ be the lower contraction of $(\Gamma,\hat{\tau_{0}})$ and let $(\Gamma'',v_{1})$ be the upper contraction of $(\Gamma',\hat{\tau_{1}})$. Finally let $\Gamma_{R}$ be the map obtained by adding another edge $e'$ to $\Gamma''$ between $v_{0}$ and $v_{1}$, on the outside of the map, so that $e$ and $e'$ are the only edges on the outer face of $\Gamma_{R}$. Since $v_{0}$ and $v_{1}$ are contracted vertices in $\Gamma_{R}$, and the outer face has degree $2$, $\Gamma_{R}$ is counted by $R$. Since there are at least two edges between $v_{0}$ and $v_{1}$ in $\Gamma_{R}$, this map cannot contain only a single edge. However, for any other map $\Gamma_{R}$ counted by $R$, the transformations between $\Gamma$ and $\Gamma_{R}$ can be reversed, so $\Gamma_{R}$ can be any other map counted by $R$. Hence, the possible maps $\Gamma_{R}$ are counted by
\[R(x,y,z)-xyz,\]
where $x$ counts the edges, $y$ counts the degree of $v_{0}$ and $z$ counts the degree of $v_{1}$. Since the transformation from $\Gamma_{R}$ to $\Gamma''$ just removes one edge between $v_{0}$ and $v_{1}$, the possible maps $\Gamma'$ are counted by
\[\frac{1}{xyz}(R(x,y,z)-xyz)=\frac{1}{xyz}R(x,y,z)-1.\]
Since $(\Gamma',\hat{\tau_{1}})$ can be any upper expansion of $(\Gamma'',v_{1})$, the possible graphs $\Gamma'$ are counted by
\[\Omega_{z}\left(\frac{1}{xyz}R(x,y,z)-1\right),\]
where $x$ counts the edges in $\Gamma'$ and $y$ counts the degree of $v_{0}$. Similarly, since $(\Gamma,\hat{\tau_{0}})$ can be any upper expansion of $(\Gamma',v_{0})$, the possible graphs $\Gamma$ are counted by
\[V(x)=\Omega_{y}\left(\Omega_{z}\left(\frac{1}{xyz}R(x,y,z)-1\right)\right).\]
\end{proof}

\section{The algorithms}
From these functional equations, we use a dynamic program to calculate the coefficients in polynomial time. For the case of general rooted planar Eulerian orientations, this is possible, since if we calculate the coefficient of $x^n$ in each of the functions $T,S,R,H,F$ in that order, for $n=0,1,2,\dots$, then each of these coefficients is determined only by values which have been previously calculated. The coefficients were calculated {\em modulo} a prime smaller than $2^{31},$ repeated for several different primes, sufficient to calculate the coefficient by use of the Chinese Remainder Theorem. In this way we calculated 90 terms  of the generating function for planar Eulerian orientations counted by edges $U(x),$ and 100 terms for the generating function for 4-valent  planar Eulerian orientations counted by vertices, $A(x).$

\section{Analysis of generating functions}
We first tried to analyse these series by the method of differential approximants (DAs)  \cite{GJ72, G89, GJ09}. The results were not totally straightforward. Assuming a power-law singularity of the form $$f(x) \sim C(1-x/x_c)^\alpha,$$ then for $U(x)$ we found the closest singularity to the origin to be at $x_c \approx 0.07957736,$ with an exponent around $\alpha \approx 1.24.$ However there was a second singularity very close by, at $x \approx 0.0795789,$ with an exponent around $2.26,$ and a third, less precisely located singularity at around $x \approx 0.0798,$ with a complex exponent the value of which is irrelevant. 

For $A(x)$ we found essentially identical results, just with a changed radius of convergence. In particular we found the closest singularity to the origin to be at $x_c \approx 0.04594404,$ with an exponent around $1.23.$ There was a second singularity very close by, at $x \approx 0.04594449,$ with an exponent around $2.23,$ and a third, less precisely located, at around $x \approx 0.0459,$ with a complex exponent the value of which is irrelevant. 

This behaviour, where one has two singularities very close together, with an exponent separated by about 1.0, is known to be characteristic of a confluent singularity, and more precisely, a confluent singularity involving a logarithmic term. To illustrate this explicitly, we constructed a test series, chosen by our expectation that series $A(x)$ at least should be of the form given in the introduction, being derived from Kostov's \cite{K00} solution of the six-vertex model. It is  $$f(x)=\frac{-x(1-\mu x)}{\log(1-\mu x)},$$ where, anticipating our later results, we take $\mu = 4\pi=12.56637061435917=1/0.0795774715459476678844418.$ Then $$[x^n]f(x)=c\cdot \mu^n/(n^2\cdot \log^2{n}),$$ \cite{FS09}. We analysed the series with 3rd order differential approximants, using a series of length 50 terms, i.e. up to $O(x^{50}).$ 

This function is not D-finite, and is not well-represented by DAs. Indeed the DAs are found to have two very close singularities, the most precisely located one is at $0.07957733,$ with an exponent $1.30-1.32,$ the other is at $0.0795782-0.0795786,$ with exponent around $2.30$ plus a nearby third singularity, much less precisely located, at  around $0.07955-0.07958$ with an exponent of $2.5-3.5$ plus a small imaginary component. That is to say, very similar behaviour to that observed above for  $A(x)$ and $U(x).$ 

The similarity in behaviour of the test series and both the series $A(x)$ and $U(x)$ is very suggestive. Indeed, it is on this basis that we were led to conjecture that the radius of convergence of $U(x)$ is $1/(4\pi).$ Similarly, the radius of convergence of $A(x)$ is conjectured to be $1/(4\sqrt{3} \pi).$

Having seen that the method of DAs has difficulties in estimating the critical exponent in this case, we turned to ratio-based methods.

If the generating function behaves as in our test series, then
\begin{equation}\label{eqn:coeff}
[x^n]f(x) = \frac{c\cdot \mu^n}{n^2}\left(\frac{1}{\log^2{n}}+\frac{a}{\log^3{n}} +\frac{b}{\log^4{n}} + \frac{c}{\log^5{n}} + o\left ( \frac{1}{\log^5{n}} \right ) \right ).
\end{equation}
 To extract asymptotics from numerical data is difficult when successive terms are only weaker by a factor of a logarithm, which varies but slowly unless one has a vast number of terms.

 The ratio of successive coefficients in this case behaves as  
 $$r_n = \frac{[x^n]f(x)}{[x^{n-1}]f(x)} = \mu \left (1 - \frac{2}{n} - \frac{2}{n\log{n}} \left (1+ \frac{c_1}{\log{n}} +  \frac{c_2}{\log^2{n}} +  \frac{c_3}{\log^3{n}}  \right ) + o\left ( \frac{1}{n\log^4{n}} \right ) \right ). $$
We show in figures \ref{fig:prat1} and \ref{fig:vrat1} the ratios for $U(x)$ and $A(x)$ plotted against $1/n.$ Both plots exhibit slight concavity, due to the logarithmic corrections. 

\begin{figure}[ht]
\setlength{\captionindent}{0pt}
\begin{minipage}{0.48\textwidth}
   \includegraphics[width=0.97\linewidth]{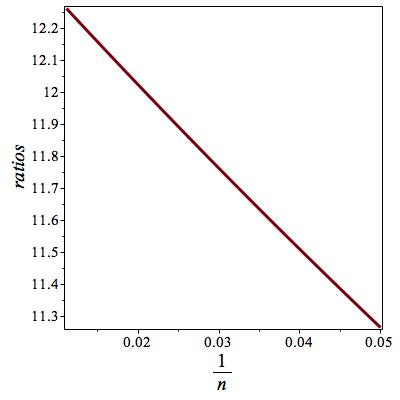} 
   \caption{Ratio plot of coefficients of $U(x).$ }
   \label{fig:prat1}
\end{minipage}\hfill
\begin{minipage}{0.48\textwidth}
   \includegraphics[width=0.97\linewidth]{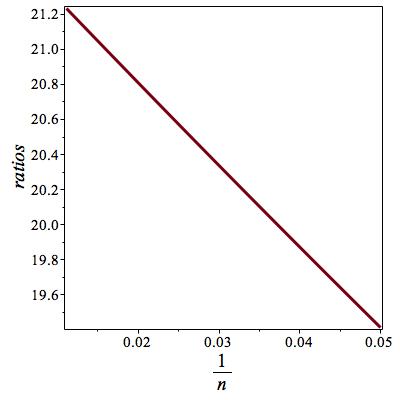} 
   \caption{Ratio plot of coefficients of $A(x).$  }
   \label{fig:vrat1}
\end{minipage}
\end{figure}

If we eliminate the O$(1/n)$ term by constructing linear intercepts, \begin{equation}\label{eqn:li}
l_n = n\cdot r_n - (n-1) \cdot r_{n-1} = \mu \left ( 1 + \frac{2}{n\log^2{n}} +\frac{4c_1}{n\log^3{n}}+\frac{6c_2}{n\log^4{n}}+o\left (\frac{1}{n\log^4{n}} \right ) \right )
\end {equation}
the corresponding plots of the linear intercepts against $1/(n\cdot \log^2{n})$ are shown in figures  \ref{fig:plin1} and \ref{fig:vlin1}. Note that the ordinate is compressed by about a factor of 10, and secondly, the plot exhibits {\em more} curvature, presumably reflecting competition between subdominant logarithmic terms. Indeed, from the asymptotics, it is clear that this sequence must eventually have a positive gradient as $n$ increases, so must pass through a maximum. We will see below that this occurs for sufficiently large $n.$

\begin{figure}[ht]
\setlength{\captionindent}{0pt}
\begin{minipage}{0.48\textwidth}
   \includegraphics[width=0.97\linewidth]{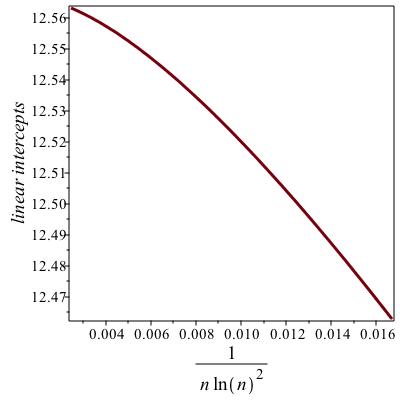} 
   \caption{Plot of linear intercepts of ratios of $U(x)$ vs. $1/n\log^2{n}.$ }
   \label{fig:plin1}
\end{minipage}\hfill
\begin{minipage}{0.48\textwidth}
   \includegraphics[width=0.97\linewidth]{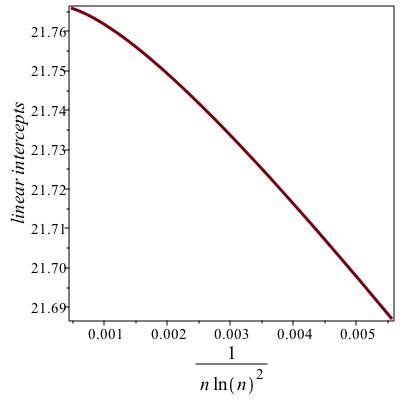} 
   \caption{Plot of linear intercepts of ratios of $A(x)$ vs. $1/n\log^2{n}.$  }
   \label{fig:vlin1}
\end{minipage}
\end{figure}

These ratios are behaving quite smoothly, and it would be desirable to have many more. It is not realistic to get vastly more terms exactly, but we can get them approximately with high enough precision for our purposes by using the method of series extension.
The idea behind this method to obtain further ratios (or terms) is simply to use the method of differential approximants {\em to predict subsequent ratios/terms}. The detailed description as to how this is done is given in \cite{G16}. 

Suffice it to say, every differential approximant naturally reproduces exactly all coefficients used in its derivation, and, being a D-finite differential equation, which implies the existence of a linear recurrence for the coefficients, therefore implies the value of {\em all} subsequent coefficients. These subsequent coefficients will not be exact (unless the solution is D-finite of sufficiently low degree that the DA is exact), but are approximate. It is to be expected that the first approximate coefficient will be the most accurate, while the accuracy will decline with increasing order of predicted coefficients. In practice we construct many DAs. We then calculate the average of the predicted coefficients (or ratios) across all constructed DAs, as well as the standard deviation, and have experimentally found the true error to be between 1 and 2 standard deviations. 

The number of terms we can predict varies from problem to problem. In this case we are extremely fortunate, in that the standard deviation of the coefficient estimates increases extremely slowly, and so we are confident in predicting 1000 extra ratios for both series which we expect to be accurate to more than 10 significant digits. That is more than enough for our purposes. Using these additional terms, we reconstruct the plot shown in Figure \ref{fig:vlin1} in Figure \ref{fig:vlin1000-1}, using a further 1000 ratios. Note that the locus passes through a maximum, reflecting competition between the subdominant logarithmic terms, and the linear intercepts are now decreasing with increasing $n,$ as predicted by the asymptotic expression (\ref{eqn:li}). In Figure \ref{fig:vlin1000-2}, we show the same plot, but with the abscissa restricted to ratios corresponding to $700 \le n \le 1100.$ The value of the ordinate at the origin in Figure \ref{fig:vlin1000-2}  is precisely $4\sqrt{3}\pi,$ and the extrapolated locus is convincingly going through the origin.

\begin{figure}[ht]
\setlength{\captionindent}{0pt}
\begin{minipage}{0.48\textwidth}
   \includegraphics[width=0.97\linewidth]{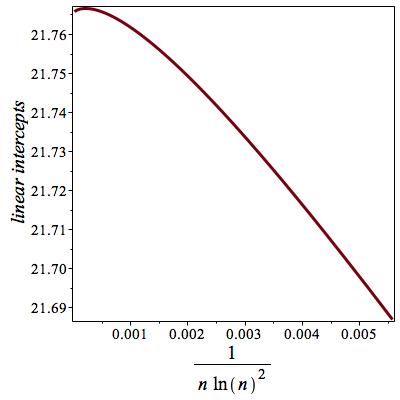} 
   \caption{Plot of linear intercepts of ratios of $A(x)$ vs. $1/n\log^2{n},$ using an extra 1000 ratios. }
   \label{fig:vlin1000-1}
\end{minipage}\hfill
\begin{minipage}{0.48\textwidth}
   \includegraphics[width=0.97\linewidth]{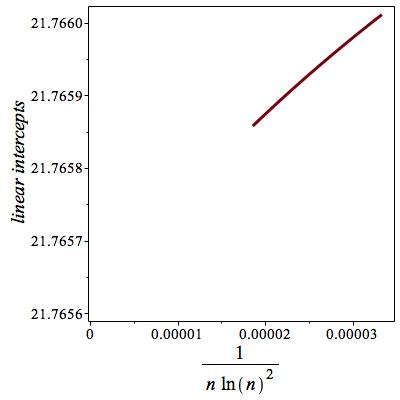} 
   \caption{Plot of linear intercepts of ratios of $A(x)$ vs. $1/n\log^2{n},$ using ratios 700 to 1100.  }
   \label{fig:vlin1000-2}
\end{minipage}
\end{figure}

The corresponding plots for planar orientations, given by the generating function $U(x),$ is shown in figures \ref{fig:plin1000-1} and \ref{fig:plin1000-2}, where now the value of the ordinate at the origin in Figure \ref{fig:plin1000-2} is precisely $4\pi,$ and again the extrapolated locus is convincingly going through the origin.

\begin{figure}[ht]
\setlength{\captionindent}{0pt}
\begin{minipage}{0.48\textwidth}
   \includegraphics[width=0.97\linewidth]{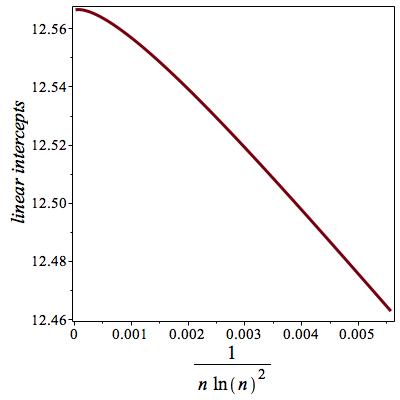} 
   \caption{Plot of linear intercepts of ratios of $U(x)$ vs. $1/n\log^2{n},$ using an extra 1000 ratios. }
   \label{fig:plin1000-1}
\end{minipage}\hfill
\begin{minipage}{0.48\textwidth}
   \includegraphics[width=0.97\linewidth]{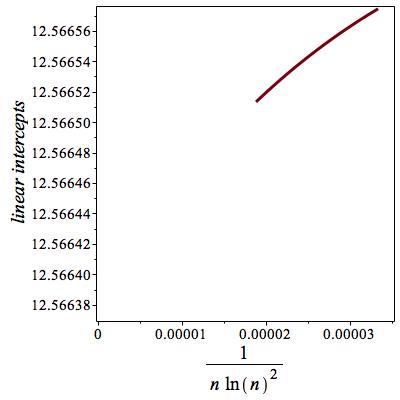} 
   \caption{Plot of linear intercepts of ratios of $U(x)$ vs. $1/n\log^2{n},$ using ratios 700 to 1100.  }
   \label{fig:plin1000-2}
\end{minipage}
\end{figure}

Now that we have good grounds to conjecture the exact value of the critical points, we are in a better position to estimate the exponent. From \cite{FS09}, p.385  we see that if
$$f(x) = (1-\mu \cdot x)^{-\alpha}\left ( \frac{1}{\mu \cdot x} \log \frac{1}{1-\mu \cdot x} \right )^\beta,$$ then 
$$[x^n]f(x) = \frac{ \mu^n \cdot n^{\alpha-1}}{\Gamma(\alpha)}(\log{n})^\beta \left(1+\frac{c_1}{\log{n}}+\frac{c_2}{\log^2{n}} +\frac{c_3}{\log^3{n}} + \frac{c_4}{\log^4{n}} +o\left ( \frac{1}{\log^4{n}} \right )  \right ),$$ where $$c_k={\beta \choose k}\Gamma{(\alpha)} \frac{d^k}{ds^k}\frac{1}{\Gamma(s)}\bigg\rvert_{s=\alpha}.$$ When $\alpha$ is a negative integer, the evaluation of the constants must be interpreted as a limiting case as the $\Gamma$ function diverges, so that certain constants vanish. In particular, provided that $\alpha$ is a negative integer and $\beta$ is not zero or a positive integer, one has
$$[x^n]f(x) = { \mu^n \cdot n^{\alpha-1}}(\log{n})^\beta \left(\frac{c_1}{\log{n}}+\frac{c_2}{\log^2{n}} +\frac{c_3}{\log^3{n}} + \frac{c_4}{\log^4{n}} +o\left ( \frac{1}{\log^4{n}} \right ) \right ).$$

The ratio of successive coefficients is in the general case $$r_n = \frac{[x^n]f(x)}{[x^{n-1}]f(x)} = \mu \left (1 + \frac{\alpha-1}{n} + \frac{\beta}{n\log{n}}- \frac{c_1}{n\log^2{n}} +  o\left ( \frac{1}{n\log^2{n}} \right )   \right ), $$ but in the case that $\alpha$ is a negative integer and $\beta$ is not zero or a positive integer, one has $$r_n = \frac{[x^n]f(x)}{[x^{n-1}]f(x)} = \mu \left (1 + \frac{\alpha-1}{n} + \frac{\beta-1}{n\log{n}}+ \frac{d}{n\log^2{n}} + o\left ( \frac{1}{n\log^2{n}} \right )  \right ), $$ where $d=-c_2/c_1.$

So one can estimate $\alpha$ from the sequence $$\alpha_n = \left ( \frac{r_n}{\mu} -1 \right )\cdot n +1  = \alpha +\frac{\beta}{\log{n}}-\frac{d}{\log^2{n}} + o\left ( \frac{1}{\log^2{n}} \right ) .$$ Plots of $\alpha_n$ against $1/\log{n}$ for both $U(x)$ and $A(x)$ respectively are shown in figures \ref{fig:palpha} and \ref{fig:valpha}, and it can be seen that having many more than 100 terms is essential. In fact the minimum in both plots occurs at around $n=100,$ and it is only with our extended data that the limit $\alpha = -1$ becomes plausible.

\begin{figure}[ht]
\setlength{\captionindent}{0pt}
\begin{minipage}{0.48\textwidth}
   \includegraphics[width=0.97\linewidth]{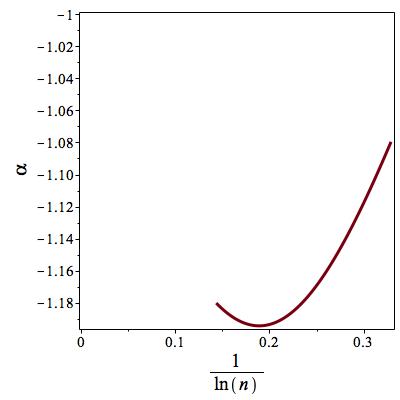} 
   \caption{Plot of exponent $\alpha$ estimates from $U(x)$ vs. $1/\log{n},$ using an extra 1000 ratios. }
   \label{fig:palpha}
\end{minipage}\hfill
\begin{minipage}{0.48\textwidth}
   \includegraphics[width=0.97\linewidth]{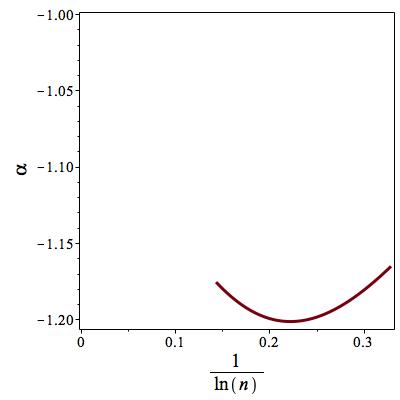} 
   \caption{Plot of exponent $\alpha$ estimates from $A(x)$ vs. $1/\log{n},$ using an extra 1000 ratios.}
   \label{fig:valpha}
\end{minipage}
\end{figure}

To take into account higher-order terms in the asymptotics, we attempted a linear fit to the assumed form (also assuming $\alpha$ is a negative integer, otherwise $\beta$ replaces $\beta-1$),
\begin{equation}\label{eqn:extrap}
\left ( \frac{r_n}{\mu} -1 \right )\cdot n +1  = \alpha +\frac{\beta-1}{\log{n}}-\frac{d}{\log^2{n}} +o\left ( \frac{1}{\log^2{n}} \right ) .
\end{equation}
 We did this by solving the linear system given by setting $n=m-1, \,\, n=m, \,\, n=m+1$ in the preceding equation, and solving for $\alpha,\,\, \beta, \,\, d,$
with $m$ ranging from 20 to the maximum possible value 1100. We obtain an $m$-dependent sequence of estimates of the terms $\alpha,\,\, \beta, \,\, d,$ which we show plotted against appropriate powers of $1/m.$ 
These are shown in figures \ref{fig:palpha1} and \ref{fig:pbeta} for planar orientations. (The corresponding plots for 4-valent orientations are similar in appearance, so are not shown).

In this way we see that both $\alpha$ and $\beta$ are plausibly going to $-1,$ as appropriate for a singularity of the form $$\frac{c\cdot\mu \cdot x \cdot(1-\mu \cdot x)}{\log(1 - \mu \cdot x)}.$$

\begin{figure}[ht]
\setlength{\captionindent}{0pt}
\begin{minipage}{0.48\textwidth}
   \includegraphics[width=0.97\linewidth]{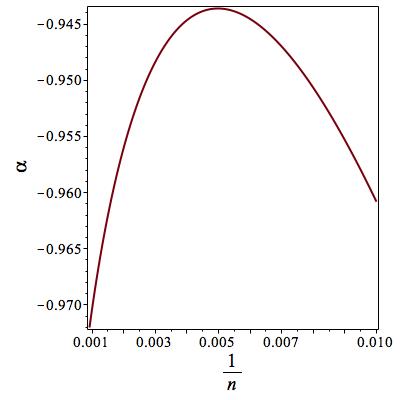} 
   \caption{Plot of exponent $\alpha$ estimates from eqn (\ref{eqn:extrap}). }
   \label{fig:palpha1}
\end{minipage}\hfill
\begin{minipage}{0.48\textwidth}
   \includegraphics[width=0.97\linewidth]{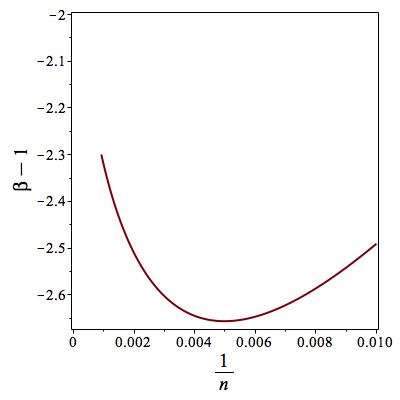} 
   \caption{Plot of exponent $\beta-1$ estimates from eqn (\ref{eqn:extrap}).}
   \label{fig:pbeta}
\end{minipage}
\end{figure}

Finally, if we accept that $\alpha = -1,$ we can refine the estimate of $\beta,$ since in that case
 \begin{equation}\label{eqn:beta2}
 \left( \frac{r_n}{\mu}-1+\frac{2}{n}\right ) n\log{n} = \beta-1+O\left (\frac{1}{\log{n}}\right ).
 \end{equation}
 The result is shown in Figure \ref{fig:pbeta2} which is plausibly tending to $\beta=-1,$ though the fact that the abscissa is $1/\log{n}$ means that one would really need many more terms, around 22,000, even to get to $0.1$ on the abscissa.

 \begin{figure}[htbp] 
 \centering
   \includegraphics[width=3in]{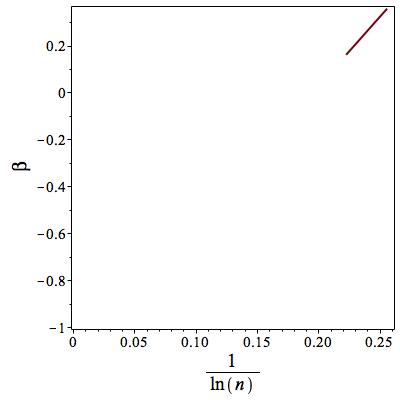} 
   \caption{Plot of exponent $\beta$ estimates from eqn (\ref{eqn:beta2}).}
   \label{fig:pbeta2}
\end{figure}

\section{Conclusion}
 We have derived a system of functional equations characterising the ordinary generating functions $A(x)$ and $U(x)$ for 4-valent planar Eulerian orientations counted by vertices and for planar Eulerian orientations counted by edges respectively. We have then developed a dynamic programming algorithm to generate coefficients of $A(x)$ and $U(x)$ of length 100 and 90 terms respectively. We then used the method of series extension to generate a further 1000 terms in each case with an accuracy of, we believe, at least 10 significant digits.

We analysed the exact and extended series in order to estimate the asymptotics. 
We found that $$A(x) \sim const.  (1 - \mu_4 z)/\log(1 - \mu_4 z)$$ and $$U(x) \sim const.  (1 - \mu z)/\log(1 - \mu z),$$ where we conjecture that $\mu_4 = 4\sqrt{3} \pi \approx 21.76559$ and  that $\mu = 4 \pi \approx 12.56637.$ 

Given this proposed asymptotic form, the generating function cannot be D-finite. Attempts to discover D-algebraic solutions from the known exact coefficients have been unsuccessful, but this could well be because we have insufficient terms.

Nevertheless, being able to conjecture the exact value of the growth constants is quite remarkable, and suggests that the problems may be exactly solvable.

After completion of this work, we realised that our conjecture for 4-valent Eulerian orientations is, with hindsight, found in the work of Kostov \cite{K00}, though the different notation, and field theoretic methods used there make the connection difficult to see. In Kostov's language, one restricts to the special case of the 6-vertex model known as the F-model (a restriction in which the weights of the different types of vertices are all equal), and sets the parameter $\lambda =\frac{1}{3},$ (see eqn. (2.1) in \cite{K00}), then from eqns. (2.10) and (3.40) the critical temperature is predicted to be $T*=4\sqrt{3}\pi.$

\section{Acknowledgements}
We wish to thank Mireille Bousquet-M\'elou for introducing us to this problem and for many stimulating discussions on this topic. In particular we thank her for pointing out the simplified version of the system of equations for the 4-valent case. We also wish to thank the referees for their critical reading of the manuscript which resulted in a significantly clearer presentation. AEP wishes to thank MASCOS and ACEMS for financial support through a PhD top-up scholarship.


\end{document}